\documentclass{amsart}

\input xy
\xyoption{all}

\newcommand{\IR}{\mathbb R}
\newcommand{\U}{\mathcal U}
\newcommand{\V}{\mathcal V}
\newcommand{\F}{\mathcal F}
\newcommand{\IQ}{\mathbb Q}
\newcommand{\Ra}{\Rightarrow}
\newcommand{\e}{\varepsilon}
\newcommand{\IN}{\mathbb N}
\newcommand{\IU}{\mathbb U}
\newcommand{\cl}{\mathrm{cl}}
\newcommand{\intr}{\mathrm{int}}

\newcommand{\upa}{\mathrm{\uparrow}}
\newcommand{\w}{\omega}
\newcommand{\dist}{\mathsf{dist}}

\newtheorem{theorem}{Theorem}[section]
\newtheorem{lemma}[theorem]{Lemma}
\newtheorem{corollary}[theorem]{Corollary}
\newtheorem{claim}[theorem]{Claim}
\newtheorem{problem}[theorem]{Problem}
\newtheorem{proposition}[theorem]{Proposition}
\theoremstyle{definition}
\newtheorem{definition}[theorem]{Definition}
\newtheorem{example}[theorem]{Example}
\newtheorem{remark}[theorem]{Remark}

\title[Quasi-metrics, quasi-uniformities and topological monoids]{Quasi-pseudometrics on quasi-uniform spaces and\\ quasi-metrization of topological monoids}
\author{Taras Banakh and Alex Ravsky}
\address{T.Banakh: Ivan Franko National University of Lviv (Ukraine), and Jan Kochanowski University in Kielce (Poland)}
\email{t.o.banakh@gmail.com}
\address{A.Ravsky: Pidstryhach Institute for Applied Problems of Mechanics and Mathematics of National Academy of Sciences, Lviv, Ukraine}
\email{oravsky@mail.ru}
\thanks{The first author has been partially financed by NCN grant DEC-2012/07/D/ST1/02087.}

\begin{document}
\begin{abstract} We define a notion of a rotund quasi-uniform space and describe a new direct construction of a (right-continuous) quasi-pseudometric on a (rotund) quasi-uniform space. This new construction allows to give alternative proofs of several classical metrizability theorems for (quasi-)uniform spaces and also obtain some new metrizability results. Applying this construction to topological monoids with open shifts, we prove that the topology of any (semiregular) topological monoid with open shifts is generated by a family of (right-continuous) left-subinvariant quasi-pseudometrics, which resolves an open problem posed by Ravsky in 2001. This implies that a topological monoid with open shifts is completely regular if and only if it is semiregular. Since each paratopological group is a topological monoid with open shifts these results apply also to paratopological groups.
\end{abstract}
\maketitle

By the classical theorem of Birkhoff and Kakutani, each first countable Hausdorff topological group is metrizable by a left-invariant metric. In \cite{Rav1} the second author found a counterpart of this result in the class of paratopological groups proving that the topology of each first countable paratopological group $G$ is generated by a left-invariant quasi-pseudometric $d$. However, the quasi-pseudometric $d$ is this result need not be continuous, since the topology on $G$ generated by a (right-)continuous quasi-metric necessarily is Tychonoff. In \cite[Question~3.1]{Rav1} the second author asked if the topology of any Tychonoff first countable paratopological group can be generated by a continuous left-invariant quasi-metric. Recently, this question of Ravsky was answered in negative by Liu \cite{Liu} who proved that if the topology of a paratopological group $G$ is generated by a left-invariant left-continuous quasi-metric, then $G$ is a topological group. In this paper we shall give a partial ``right'' answer to the question of Ravsky's proving that the topology of any regular first-countable paratopological group $G$ is generated by a left-invariant right-continuous quasi-pseudometric, more precisely, by a left-invariant quasi-pseudometric $d$ such that for every non-empty subset $A\subset G$ the distance function $\overline{d}_A:G\to [0,\infty)$, $\overline{d}_A:x\mapsto\inf\{\e>0:x\in \overline{B_d(A,\e)}\}$ is continuous. Quasi-pseudometrics with this property will be called $\overline{\dist}$-continuous. In fact, our argument works in a much more general setting of quasi-uniform spaces, which allows us to prove some new results on  quasi-pseudometrizability of quasi-uniform spaces.

The paper is organized as follows. Section~\ref{s1} is of preliminary character and collects the necessary information on topological spaces, separation axioms, metrics and their generalizations, distance functions, entourages, and quasi-uniformities. In this section we also introduce an important notion of a (point-)rotund quasi-uniform space and prove that all uniform spaces are rotund. The main result of Section~\ref{s2} is Theorem~\ref{main}, which is technically the most difficult result of the paper. In this theorem we give a new construction of a uniform quasi-pseudometric with small balls on a quasi-uniform space. Our construction is direct and differs from the classical construction tracing back to Alexandroff, Urysohn \cite{AU}, Chittenden \cite{Chit}, Frink \cite{Frink} who  used infima over chains for obtaining the triangle inequality. One of advantages of our construction is the possibility to control closures of balls of constructed quasi-pseudometrics, which results in their right-continuity.

In Section~\ref{s3} we apply Theorem~\ref{main} to give alternative proofs of some classical results on metrizability of (quasi-) uniform spaces due to Alexandroff, Urysohn \cite{AU}, Chittenden \cite{Chit}, Frink \cite{Frink}, Aronszajn \cite{Aron}, and Weil \cite{Weil}. We also derive some new result of quasi-pseudometrizability of rotund quasi-uniform spaces by (families of) $\dist$-continuous quasi-pseudometrics.

In Section~\ref{s4} we introduce point-rotund topological spaces (as topological spaces homeomorphic to point-rotund quasi-uniform spaces) and apply Theorem~\ref{main} to prove that for point-rotund spaces the complete regularity is equivalent to the semiregularity and the functional Hausdorffness is equivalent to the semi-Hausdorffness. By a different method these equivalences were first proved by the authors in \cite{BR}. The obtained results show that for point-rotund topological spaces the diagram describing the interplay between separation axioms simplifies to a nice symmetric form.

In Section~\ref{s5} we study four natural quasi-uniformities ($\mathcal L$, $\mathcal R$, $\mathcal L\vee\mathcal R$, $\mathcal L\wedge\mathcal R$) on topological monoids with open shifts and prove that three of them (namely, $\mathcal L$, $\mathcal R$, and $\mathcal L\wedge\mathcal R$) are rotund. In Section~\ref{s6} we apply Theorem~\ref{main} to construct left-subinvariant ($\dist$-continuous) quasi-pseudometrics on topological monoids with open shifts. Since each paratopological group is a topological monoid with open shifts, the results of Section~\ref{s6} apply to topological groups, which is done in Section~\ref{s7}. In particular, we prove that the topology of any first-countable (semiregular) paratopological group is generated by a (right-continuous) quasi-pseudometric, which answers a problem posed by Ravsky in \cite[Question 3.1]{Rav1}.

\section{Preliminaries}\label{s1}

In this section we recall some known information on topological spaces, quasi-uniformities, and various generalizations of metrics. Also we shall introduce rotund quasi-uniform spaces, which play a crucial role in our subsequent considerations.

\subsection{Topological spaces} For a subset $A\subset X$ of a topological space $X$ by $\cl_X(A)$ and $\intr_X(A)$ we denote the closure and the interior of $A$ in $X$, respectively.
The sets $\cl_X(A)$, $\intr_X(A)$, and $\intr_X\cl_X(A)$ will be also alternatively denoted by $\overline{A}$, $A^\circ$, and $\overline{A}^\circ$, respectively.
A subset $U$ of a topological space $X$ is called {\em regular open} if $U=\overline{U}^\circ$.

To avoid a possible ambiguity, let us recall the definitions of separation axioms we shall work with.

\begin{definition}
A topological space $X$ is called
\begin{itemize}
\item[$T_0$:] a {\em $T_0$-space} if  for any distinct points $x,y\in X$ there is an open set $U\subset X$ containing exactly one of these points;
\item[$T_1$:] a {\em $T_1$-space} if for any distinct points $x,y\in X$ the point $x$ has a neighborhood $U_x\subset X$ such that $y\notin U_x$;
\item[$T_2$:] {\em Hausdorff\/}   if for any distinct points $x,y\in X$ the point $x$ has a neighborhood $U_x\subset X$ such that $y\notin\overline{U}_x$;
\item[$T_{\kern-1pt\frac12 2}$:] {\em semi-Hausdorff\/} if for any distinct points $x,y\in X$ the point $x$ has a neighborhood $U_x\subset X$ such that $y\notin\overline{U}^\circ_x$;
\item[$T_{2\frac12}$:] {\em functionally Hausdorff\/}  if for any distinct points $x,y\in X$ there is a continuous function $f:X\to[0,1]$ such that $f(x)\ne f(y)$;
\smallskip

\item[$R$:] {\em regular} if for any point $x\in X$ and neighborhood $O_x\subset X$ of $x$ there is a neighborhood $U_x\subset X$ of $x$ such that $\overline{U}_x\subset O_x$;
\item[$\tfrac12 R$:] {\em semiregular} if for any point $x\in X$ and  neighborhood $O_x\subset X$ of $x$ there is a neighborhood $U_x\subset X$ of $x$ such that $\overline{U}^\circ_x\subset O_x$;
\item[$R\tfrac12$:] {\em completely regular} if for any point $x\in X$ and  neighborhood $O_x\subset X$ of $x$ there is continuous function $f:X\to[0,1]$ such that $f(x)=0$ and $f^{-1}\big([0,1)\big)\subset O_x$;
\item[$T_{3\frac12}$:] {\em Tychonoff\/} if $X$ is a completely regular $T_1$-space;
\item[$T_{3}$:] a {\em $T_3$-space} if $X$ is a regular $T_1$-space;
\item[$T_{\kern-1pt\frac12 3}$:] a {\em $T_{\frac12 3}$-space} if $X$ is a semi-regular $T_1$-space.
\end{itemize}
\end{definition}
For any topological space these separation axioms relate as follows:
$$
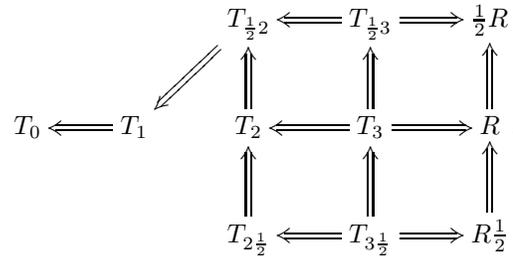
\xymatrix{
&&T_{\frac12 2}\ar@{=>}[ld]&T_{\frac12 3}\ar@{=>}[l]\ar@{=>}[r]&\tfrac12R\\
T_0&T_1\ar@{=>}[l]&T_2\ar@{=>}[u]&T_3\ar@{=>}[l]\ar@{=>}[r]\ar@{=>}[u]&R\ar@{=>}[u]&\hskip-55pt.\\
&&T_{2\frac12}\ar@{=>}[u]&T_{3\frac12}\ar@{=>}[l]\ar@{=>}[u]\ar@{=>}[r]&R\tfrac12\ar@{=>}[u]\\
}
$$
Known (or simple) examples show that none of the implications in this diagram can be reversed.

\subsection{Various generalizations of metrics}

Let us recall that a {\em metric} on a set $X$ is a function $d:X\times X\to[0,\infty)$ satisfying the axioms:
\begin{enumerate}
\item[(M1)] $d(x,x)=0$ for every $x\in X$;
\item[(M2)] for any points $x,y\in X$ the equality $d(x,y)=0$ implies that $x=y$;
\item[(M3)] $d(x,y)=d(y,x)$ for any points $x,y\in X$;
\item[(M4)] $d(x,z)\le d(x,y)+d(y,z)$ for any points $x,y,z\in X$.
\end{enumerate}
Omitting some of these axioms we obtain various generalizations of metrics. In particular, a function $d:X\times X\to[0,\infty)$ is called
\begin{itemize}
\item a {\em symmetric} if it satisfies the conditions (M1)--(M3);
\item a {\em pseudometric} if it satisfies the conditions (M1), (M3), and (M4);
\item a {\em quasi-metric} if it satisfies the conditions (M1), (M2), and (M4);
\item a {\em quasi-pseudometric} if it satisfies the conditions (M1) and (M4);
\item a {\em premetric} if it satisfies the condition (M1).
\end{itemize}
More information of these and other generalizations of metrics can be found in \cite{AP} and \cite{Hat}.

Let $d:X\times X\to[0,+\infty)$ be a premetric on a set $X$. For any point $x\in X$ and $\e>0$ let $$B_d(x;\e)=\{y\in X:d(x,y)<\e\}\mbox{ and }B_d(x;\e]=\{y\in X:d(x,y)\le\e\}$$denote the {\em open} and {\em closed $\e$-balls} centered at $x$ respectively. For a subset $A\subset X$ we put $B_d(A;\e)=\bigcup_{a\in A}B_d(a;\e)$ be the $\e$-neighborhood of $A$ in $X$.

We shall say that a premetric $d$ on a topological space $X$
\begin{itemize}
\item has {\em open balls} if for every $x\in X$ and $\e>0$ the open $\e$-ball $B_d(x;\e)$ is open in $X$;
\item has {\em closed balls} if for every $x\in X$ and $\e>0$ the closed $\e$-ball $ B_d(x;\e]$ is closed in $X$;
\item has {\em open and closed balls} it has both open balls and closed balls.
\end{itemize}
It is easy to see that a premetric $d:X\times X\to[0,\infty)$ on a topological space $X$ has open and closed balls if and only if $d$ is
right-continuous.

A premetric $d:X\times X\to[0,\infty)$ on a topological space $X$ is defined to be
\begin{itemize}
\item {\em right-continuous} if for every $x_0\in X$ the function $d(x_0,\cdot):X\to[0,\infty)$, $d(x_0,\cdot):x\mapsto d(x_0,x)$, is continuous;
\item {\em left-continuous} if for every $x_0\in X$ the function $d(\cdot,x_0):X\to[0,\infty)$, $d(\cdot,x_0):x\mapsto d(x,x_0)$, is continuous;
\item {\em separately continuous} if $d$ is left-continuous and right-continuous;
\item {\em continuous} if $d$ is continuous as a map from $X\times X$ to $[0,\infty)$.
\end{itemize}
It is easy to show that a pseudometric $d$ on a topological space $X$ is continuous if and only if it is left-continuous or right-continuous if and only if it has open balls.

We say that a family $\mathcal D$ of premetrics on a topological space $X$ generates the topology of $X$ if the family $\{B_d(x,\e):d\in\mathcal D,\;x\in X,\;\e>0\}$ is a subbase of the topology of $X$.
If the family $\mathcal D$ consists of a single premetric  $d$, then we will say that the topology of $X$ is generated by the premetric $d$.

It is known (\cite{Ku1}, \cite{Ku2}) that the topology of any space can be generated by a family of quasi-pseudometrics.

\subsection{Distance functions}
Given a premetric $d:X\times X\to[0,\infty)$ on a topological space $X$ and a non-empty subset $A\subset X$ let us consider three distance functions
$$
\begin{aligned}
&d_A:X\to [0,\infty),\;\;d_A:x\mapsto \inf\{\e>0\colon x\in B_d(A,\e)\,\},\\
&\overline{d}_A:X\to [0,\infty),\;\;\overline{d}_A:x\mapsto \inf\{\e>0\colon x\in \overline{B_d(A,\e)}\,\},\mbox{ \;\;and}\\
&\overline{d}^\circ_A:X\to [0,\infty),\;\;\overline{d}^\circ_A:x\mapsto \inf\{\e>0\colon x\in B_d(A,\e)\cup\overline{B_d(A,\e)}^\circ\}.
\end{aligned}
$$
The inclusions $B_d(A,\e)\subset B_d(A,\e)\cup \overline{B_d(A,\e)}^\circ\subset \overline{B_d(A,\e)}$ holding for every $\e>0$ imply that
$$\overline{d}_A\le \overline{d}^\circ_A\le d_A.$$
If $d$ is a premetric with open balls, then $B_d(A,\e)\subset \overline{B_d(A,\e)}^\circ$ and the definition of the distance function $\overline{d}^\circ_A$ has a simpler and more natural form $\overline{d}^\circ_A(x)=\inf\{\e>0\colon x\in \overline{B_d(A,\e)}^\circ\}$.

\begin{proposition}\label{p1.2n} For a premetric $d$ with open balls on a topological space $X$ and a non-empty subset $A\subset X$ the following conditions are equivalent:
\begin{enumerate}
\item $\overline{d}_A=\overline{d}^\circ_A$;
\item the distance function $\overline{d}_A:X\to[0,\infty)$ is continuous;
\item the distance function $\overline{d}^\circ_A:X\to[0,\infty)$ is continuous.
\end{enumerate}
The equivalent conditions \textup{(1)--(3)} follows from the equivalent conditions:
\begin{itemize}
\item[(4)] the distance function $d_A$ is continuous;
\item[(5)] $\overline{d}_A=\overline{d}^\circ_A=d_A$.
\end{itemize}
\end{proposition}

\begin{proof} First we prove that (1) implies (2) and (3). Assume that  $\overline{d}_A=\overline{d}^\circ_A$. The continuity of the function  $\overline{d}_A=\overline{d}^\circ_A$ will follow as soon as we check that for any positive real number $r$ the sets $(\overline{d}^\circ_A)^{-1}\big([0,r)\big)$ and $(\overline{d}_A)^{-1}\big((r,\infty)\big)$ are open in $X$. The first set is open as $(\overline{d}^\circ_A)^{-1}\big([0,r)\big)=\bigcup_{\e<r}\overline{B_d(A,\e)}^\circ.$
The second set is open since $(\overline{d}_A)^{-1}\big((r,\infty)\big)=\bigcup_{\e>r}X\setminus \overline{B_d(A,\e)}$.

Next we shall prove that (2) or (3) implies (1).
Assume that $\overline{d}_A$ or $\overline{d}^\circ_A$ is continuous. Since $\overline{d}_A\le\overline{d}^\circ_A$, it suffices to check that  $\overline{d}_A(x)\ge\overline{d}^\circ_A(x)$ for every $x\in X$. Assume conversely that
 $\overline{d}_A(x)<\overline{d}^\circ_A(x)$ for some $x\in X$ and choose any real number $r$ such that $\overline{d}_A(x)<r<\overline{d}^\circ_A(x)$. If the function $\overline{d}_A$ is continuous, then
 the set $\overline{d}^{-1}_A\big([0,r)\big)=\bigcup_{\e<r}\overline{B_d(A,\e)}$ is open and hence is contained in $\overline{B_d(A,r)}^\circ$. Then $$x\in \overline{d}^{-1}_A\big([0,r)\big)=\bigcup_{\e<r}\overline{B_d(A,\e)}\subset\overline{B_d(A,r)}^\circ$$
 implies $\overline{d}^\circ_A(x)\le r<\overline{d}^\circ_A(x)$, which is a desired contradiction.

 If the function  $\overline{d}^\circ_A$ is continuous, then
 the set $(\overline{d}^\circ_A)^{-1}\big((r,\infty)\big)=\bigcup_{\e>r}X\setminus \overline{B_d(A,\e)}^\circ\subset X\setminus \overline{B_d(A,r)}^\circ$ is open and hence is contained in $\intr_X(X\setminus \overline{B_d(A,r)}^\circ)=X\setminus\cl_X(\overline{B_d(A,r)}^\circ)=X\setminus\overline{B_d(A,r)}.$
 Then
 $$x\in (\overline{d}^\circ_A)^{-1}\big((r,\infty)\big)=\bigcup_{\e>r}X\setminus \overline{B_d(A,\e)}^\circ\subset X\setminus \overline{B_d(A,r)}$$
 implies $\overline{d}_A(x)\ge r>\overline{d}_A(x)$, which is a desired contradiction.
 \smallskip

It is clear that $(5)\Ra(1)$. The implication $(5)\Ra(4)$ follows from the equivalence (1)--(3).
It remain to prove that $(4)\Ra(5)$. Assume that the distance function $d_A$ is continuous but the equality $\overline{d}_A=\overline{d}_A^\circ=d_A$ does not hold. Taking into account that  $\overline{d}_A\le\overline{d}_A^\circ\le d_A$, we could find a point $x\in X$ such that $\overline{d}_A(x)<d_A(x)$. Choose a real number $r$ with $\overline{d}_A(x)<r<d_A(x)$.
Then $x\in\overline{B_d(A,r)}$. Observe that $B_d(A,r)\subset d_A^{-1}([0,r))\subset d_A^{-1}([0,r])$.
By the continuity of $d_A$, the set $d_A^{-1}([0,r])$ is closed and hence
$x\in\overline{B_d(A,r)}\subset d_A^{-1}([0,r])$ and $d_A(x)\le r<d_A(x)$, which is a desired contradiction showing that $\overline{d}=\overline{d}^\circ=d$.
 \end{proof}

A premetric $d:X\times X\to[0,\infty)$ will be called
{\em $\dist$-continuous} (resp. {\em $\overline{\dist}$-continuous}, {\em $\overline{\dist}^\circ$-continuous}) if for every non-empty subset $A\subset X$ the distance function $d_A:X\to[0,\infty)$ (resp. $\overline d_A$, $\overline{d}^\circ_A$) is continuous.
Since $d(x,y)=d_{\{x\}}(y)$ for $x,y\in X$, every $\dist$-continuous premetric is right-continuous.

Proposition~\ref{p1.2n} implies the following characterization:

\begin{corollary}\label{c1.3n} Let $d$ be a premetric on a topological space $X$. If $d$ has open balls, then the following conditions are equivalent:
\begin{enumerate}
\item $d$ is $\overline{\dist}$-continuous;
\item $d$ is $\overline{\dist}^\circ$-continuous;
\item $\overline{d}_A=\overline{d}^\circ_A$ for any non-empty subset $A\subset X$.
\end{enumerate}
The equivalent conditions \textup{(1)--(3)} follows from the equivalent conditions:
\begin{itemize}
\item[(4)] $d$ is $\dist$-continuous;
\item[(5)] $\overline{d}_A=\overline{d}^\circ_A=d_A$ for any non-empty subset $A\subset X$.
\end{itemize}
The equivalent conditions \textup{(4),(5)} imply:
\begin{itemize}
\item[(6)] $d$ is right-continuous.
\end{itemize}
\end{corollary}
The implications from Corollary~\ref{c1.3n} are shown on the following diagram:
$$\xymatrix{
\mbox{$\overline{\dist}$-continuous}\ar@{<=>}[r]&
\mbox{$\overline{\dist}^\circ$-continuous}&
\mbox{$\dist$-continuous}\ar@{=>}[l]\ar@{=>}[r]&
\mbox{right-continuous}.
}
$$

Having in mind that $d(x,y)=d_{\{x\}}(y)$ for any $x,y\in X$, let us define two modifications $\overline{d}$ and $\overline{d}^\circ$ of the premetric $d$ letting
$$\mbox{$\overline{d}(x,y)=\overline{d}_{\{x\}}(y)$ and $\overline{d}^\circ(x,y)=\overline{d}^\circ_{\{x\}}(y)$ for $x,y\in X$.}$$
The premetrics $\overline{d}$ and $\overline{d}^\circ$ are called the {\em regularization} and the {\em semiregularization} of $d$, respectively.

It follows that $\overline{d}\le\overline{d}^\circ\le d$. Moreover,
$$B_{\bar{d}}(x,\e)=\bigcup_{\delta<\e}\overline{B_d(x,\delta)}\mbox{ \ and \ }B_{\overline{d}^\circ}(x,\e)=\bigcup_{\delta<\e}B(x,\delta)\cup\overline{B_d(x,\delta)}^\circ\mbox{ for any $x\in X$ and $\e>0$}.$$

\begin{proposition}\label{p1.4n} Let $d:X\times X\to [0,\infty)$ be a premetric on a topological space $X$.
\begin{enumerate}
\item The regularization $\overline{d}$ of $d$ is a premetric with closed balls on $X$;
\item If $d$ has open balls, then the semiregularization $\overline{d}^\circ$ of $d$ is a premetric with open balls on $X$.
\end{enumerate}
\end{proposition}

\begin{proof} 1. It follows that for every $r\in [0,\infty)$ and $x\in X$ the set
$$X\setminus B_{\bar{d}}(x,r]=\bigcup_{\e>r}X\setminus \overline{B_d(x,\e)}$$is open, which implies that the closed ball $B_{\overline{d}}(x,r]$ is closed in $X$.
\smallskip

2. If $d$ has open balls, then  for every $r\in [0,\infty)$ and $x\in X$ the ball
$$B_{\overline{d}^\circ}(x,r)=\bigcup_{\e<r}\big(B_d(x,\e)\cup\overline{B_{d}(x,\e)}^\circ\,\big)=
\bigcup_{\e<r}\overline{B_d(x,\e)}^\circ$$is open.
\end{proof}

The operations of regularization and semiregularization are idempotent in the following sense.

\begin{proposition}\label{p1.4m} Given a premetric $d$ on a topological space $X$, consider its modifications $p=\overline{d}$ and $\rho=\overline{d}^\circ$. For any non-empty subset $A\subset X$ we get $\overline{p}_A=\overline{d}_A$ and $\overline{\rho}^\circ_A=\overline{d}^\circ_A$.
\end{proposition}

\begin{proof} First we prove the equality $\overline{p}_A=\overline{d}_A$. The inequality $p=\overline{d}\le d$ implies that $\overline{p}_A\le\overline{d}_A$. So, it suffices to prove that $\overline{p}_A\ge \overline{d}_A$. Assuming the converse, we could find a point $x\in X$ such that $\overline{p}_A(x)<\overline{d}_A(x)$. Choose any real number $r$ such that $\overline{p}_A(x)<r<\overline{d}_A(x)$. It follows from $\overline{p}_A(x)<r$ that $$x\in \overline{B_p(A,r)}=\overline{B_{\bar{d}}(A,r)}=\cl_X\Big(\bigcup_{\e<r}\overline{B_d(A,\e)}\Big)\subset \cl_X(\overline{B_d(A,r)})=\overline{B_d(A,r)}$$and hence $\overline{d}_A(x)\le r<\overline{d}_A(x)$, which is a desired contradiction.

Next, we prove that $\overline{p}^\circ_A=\overline{d}^\circ_A$. The inequality $\rho=\overline{d}^\circ\le d$ implies that $\overline{\rho}^\circ_A\le\overline{d}^\circ_A$. So, it suffices to prove that $\overline{\rho}^\circ_A\ge \overline{d}^\circ_A$. Assuming the converse, we could find a point $x\in X$ such that $\overline{\rho}^\circ_A(x)<\overline{d}^\circ_A(x)$. Choose any real number $r$ such that $\overline{\rho}^\circ_A(x)<r<\overline{d}^\circ_A(x)$. It follow from $\overline{\rho}^\circ_A(x)<r$ that $$
\begin{aligned}
x&\in B_\rho(A,r)\cup \overline{B_\rho(A,r)}^\circ= B_{\overline{d}^\circ}(A,r)\cup\overline{B_{\overline{d}^\circ}(A,r)}^\circ=\\
&=
\bigcup_{\e<r}\big(B_d(A,\e)\cup\overline{B_d(A,\e)}^\circ\big)\cup
\intr_X\cl_X\Big(\bigcup_{\e<r}B_d(A,\e)\cup\overline{B_d(A,\e)}^\circ\Big)\subset\\
 &\subset B_d(A,r)\cup \overline{B_d(A,r)}^\circ\cup \intr_X\cl_X(B_d(A,r)\cup\overline{B_d(A,r)}^\circ)=B_d(A,r)\cup\overline{B_d(A,r)}^\circ
\end{aligned}
$$and hence $\overline{d}^\circ_A(x)\le r<\overline{d}^\circ_A(x)$, which is a desired contradiction.
\end{proof}

  Proposition~\ref{p1.2n}, \ref{p1.4m} and Corollary~\ref{c1.3n} imply:

\begin{corollary}\label{c1.6n} Let $d$ be a premetric with open balls on a topological space $X$.
\begin{enumerate}
\item If $\overline{d}=\overline{d}^\circ$, then the premetric $\overline{d}=\overline{d}^\circ$ is right-continuous.
\item If the premetric $d$ is $\overline{\dist}$-continuous, then $\overline{d}=\overline{d}^\circ$ and the premetric $\overline{d}=\overline{d}^\circ$ is right-continuous and $\overline{\dist}$-continuous.
\end{enumerate}
\end{corollary}

\begin{proof}
The first statement follows from Proposition~\ref{p1.2n}. Now assume that the premetric $d$ is $\overline{\dist}$-continuous. By Corollary~\ref{c1.3n}, $\overline{d}_A=\overline{d}^\circ_A$ for every non-empty $A\subset X$. In particular, $\overline{d}_{\{x\}}=\overline{d}^\circ_{\{x\}}$ for every $x\in X$, which implies $\overline{d}=\overline{d}^\circ$. By the $\overline{\dist}$-continuity of $d$, for every $x\in X$ the distance map $\overline{d}_{\{x\}}$ is continuous, which implies that the premetric $\overline{d}$ is right-continuous. The $\overline{\dist}$-continuity of the premetric $\overline{d}$ follows from Proposition~\ref{p1.4m} and the $\overline{\dist}$-continuity of the premetric $d$.
\end{proof}

\subsection{Entourages and balls} By an {\em entourage} on a set $X$ we shall understand any subset $U\subset X\times X$ containing the diagonal $\Delta_X=\{(x,y)\in X\times X:x=y\}$.

Given two entourages $U,V$ on $X$ let
$$U\circ V=\big\{(x,z)\in X\times X:\mbox{$\exists y\in X$ such that $(x,y)\in U$ and $(y,z)\in V$}\big\}$$be their composition and
$U^{-1}=\{(y,x)\in X\times X:(x,y)\in U\}$ be the inverse entourage to $U$. We put $U^1=U$ and $U^{n+1}=U^n\circ U$ for $n\ge 0$. Sometimes it will be convenient to denote the composition $U\circ V$ by $UV$.

For an entourage $U\subset X\times X$ and a point $x\in X$ the set $B(x;U)=\{y\in X:(x,y)\in U\}$ is called the {\em $U$-ball centered} at $x$. For a subset $A\subset X$ the set $B(A;U)=\bigcup_{a\in A}B(a;U)$ is the {\em $U$-neighborhood} of $A$.

Observe that for a premetric $d$ on a set $X$ we get $B_d(x;\e)=B(x;[d]_{<\e})$ and $B_d(x;\e]=B(x;[d]_{\le\e})$where
$$[d]_{<\e}=\{(x,y)\in X\times X:d(x,y)<\e\}\mbox{ \ and \ }
[d]_{\le\e}=\{(x,y)\in X\times X:d(x,y)\le\e\}.
$$

For an entourage $U\subset X\times X$ on a topological space $X$ by $\overline{B}(x;U)$, $B^\circ\kern-1pt(x;U)$, and $\overline{B}^\circ\kern-2pt(x;U)$ we shall denote the closure, interior, and the interior of the closure of the $U$-ball $B(x;U)$ of $x$ in the topological space $X$.

Let also $$\overline{U}=\bigcup_{x\in X}\{x\}\times \overline{B}(x;U),\;\;
{U}^\circ=\bigcup_{x\in X}\{x\}\times B^\circ\kern-1pt(x;U),\mbox{ \ and \ }
\overline{U}^\circ=\bigcup_{x\in X}\{x\}\times \overline{B}^\circ\kern-2pt(x;U)$$be the closure, interior and the interior of the closure of $U$ in $X_d\times X$ where $X_d$ is the set $X$ endowed with the discrete topology.

A family of entourages $\mathcal B$ on a set $X$ is called {\em multiplicative} if for any entourages $U,V\in\mathcal B$ their composition $U\circ V$ belongs to $\mathcal B$.

\subsection{The balanced product of entourages} In this subsection we shall introduce and discuss the operation of balanced product of entourages. This operation will play a crucial role in the proof of Theorem~\ref{main}. Let $\mathcal B$ be any multiplicative family of entourages on a set $X$.

For any entourages $U_1,\dots,U_n\in\mathcal B$ define their {\em balanced product} $\Pi_n(U_1,\dots,U_n)$ by recursion letting $\Pi_1(U_1)=U_1$ and $$\Pi_{n}(U_1,U_2,\dots,U_n)=\Pi_{n-1}(U_2,\dots,U_n)\circ U_1\circ \Pi_{n-1}(U_2,\dots,U_n)\mbox{ \ for $n>1$}.$$ Since the base $\mathcal B$ is multiplicative, the entourage $\Pi_n(U_1,\dots,U_n)$ belongs to $\mathcal B$.
%In particular, $$\Pi_2(U_1,U_2)=U_2U_1U_2,\;\;\Pi_3(U_1,U_2,U_3)=U_3U_2U_3U_1U_3U_2U_3,$$ and so on.

Let $\mathcal B^{<\IN}=\bigcup_{n\in\IN}\mathcal B^n$ and define the function $\Pi:\mathcal B^{<\IN}\to \mathcal B$ letting $\Pi|\mathcal B^n=\Pi_n$ for every $n\in\IN$. The function $\Pi$ has the following associativity property:

\begin{lemma} For any numbers $n\ge k>1$ and entourages $U_1,\dots,U_n\in\mathcal B$ we get
$$\Pi(U_1,\dots,U_{n})=\Pi(U_1,\dots,U_{k-1},\Pi(U_k,\dots,U_n)).$$
\end{lemma}

\begin{proof} This lemma will be proved by induction on $k$. For $k=2$ the equality
$$
\begin{aligned}
\Pi(U_1,\dots,U_{n})&=\Pi_n(U_1,\dots,U_n)=\Pi_{n-1}(U_2,\dots,U_n)\circ U_1\circ \Pi_{n-1}(U_2,\dots,U_n)=\\
&=\Pi_2(U_1,\Pi_{n-1}(U_2,\dots,U_n))=\Pi(U_1,\Pi(U_2,\dots,U_n))
\end{aligned}
$$ follows from the definition of the function $\Pi_n$.

 Assume that for some $k>2$ and every $m\ge k$ and entourages $V_1,\dots,V_m\in\mathcal B$ the equality $$\Pi(V_1,\dots,V_{m})=\Pi_k(V_1,\dots,V_{k-1},\Pi_{m-k+1}(V_k,\dots,V_m))=
\Pi(V_1,\dots,V_{k-1},\Pi(V_k,\dots,V_m))$$ has been proved.

Then for every $n\ge k+1$ and entourages $U_1,\dots,U_n\in\U$, by the definition of the function $\Pi_{k+1}$ and the inductive hypothesis, we get
$$
\begin{aligned}
\Pi(U_1,\dots,&U_{k},\Pi(U_{k+1},\dots,U_n))=\Pi_{k+1}(U_1,\dots,U_{k},\Pi_{n-k}(U_{k+1},\dots,U_n))=\\
&=\Pi_k(U_2,\dots,U_k,\Pi_{n-k}(U_{k+1},\dots,U_n))\circ U_1\circ\Pi_k(U_2,\dots,U_k,\Pi_{n-k}(U_{k+1},\dots,U_n))=\\
&=\Pi_{n-1}(U_2,\dots,U_n)\circ U_1\circ \Pi_{n-1}(U_2,\dots,U_n)=\Pi_n(U_1,\dots,U_n)=\Pi(U_1,\dots,U_n).
\end{aligned}
$$
\end{proof}

%For entourages $U_1,\dots,U_n$ on $X$ the entourage $\Pi(U_1,\dots,U_n)$ is called the {\em balanced product} of the entourages $U_1,\dots,U_n$.

\subsection{Quasi-uniformities}

A {\em quasi-uniformity} on a set $X$ is a family $\U$ of entourages on $X$ such that
\begin{itemize}
\item[(U1)] for any $U,V\in\U$ there is $W\in\U$ such that $W\subset U\cap V$;
\item[(U2)] for any  $U\in\U$ there is  $V\in\U$ such that $V\circ V\subset U$;
\item[(U3)] for every entourage $U\in\U$, any subset $V\subset X\times X$ containing $U$ belongs to $\U$.
\end{itemize}
A quasi-uniformity $\U$ on $X$ is called a {\em uniformity} if $\U=\U^{-1}$ where $\U^{-1}=\{U^{-1}:U\in\U\}$.

A {\em quasi-uniform space} is a pair $(X,\U)$ consisting of a set $X$ and a quasi-uniformity $\U$ on $X$.

Each quasi-uniformity $\U$ on a set $X$ generates a topology $\tau_\U$ on $X$ consisting of sets $W\subset X$ such that for each $x\in W$ there is $U\in\U$ such that $B(x;U)\subset W$. It can be shown that for every point $x\in X$ the family $\{B(x;U)\}_{U\in\U}$ is a neighborhood base at $x$.
Talking about topological properties of a quasi-uniform space $(X,\U)$ we shall always have in mind the topology $\tau_\U$. Given a quasi-uniformity $\U$ on a topological space $(X,\tau)$ we shall always assume that $\tau_\U\subset\tau$.

It is known (see \cite{Ku1} or \cite{Ku2}) that the topology of any topological space $X$ is generated by a suitable quasi-uniformity.

%Saying that $\U$ is a quasi-uniformity on a topological space $(X,\tau)$ we shall always assume that $\tau_\U\subset\tau$, i.e., for every $x\in X$ and $U\in\U$ the ball $B(x;U)$ is a neighborhood of $x$ in $(X,\tau)$.

Let $\U$ be a quasi-uniformity on a set $X$. A subfamily $\mathcal B\subset\U$ is called
\begin{itemize}
\item a {\em base} of a quasi-uniformity $\U$ if each entourage $U\in\U$ contains some basic entourage $B\in\mathcal B$;
\item a {\em subbase} of a quasi-uniformity $\U$ if the family $\{\cap\F:\F\subset\mathcal B,\;\;|\F|<\w\}$ is a base of $\U$.
\end{itemize}
It is easy to see that each base of a quasi-uniformity satisfies the axioms (U1), (U2), and each family $\mathcal B$ of entourages on $X$ satisfying the axioms (U1), (U2) is a base of a unique quasi-uniformity $\U$, called the  quasi-uniformity generated by the base $\mathcal B$.

By the {\em uniform character} $\chi(\U)$ of a quasi-uniformity $\U$ we understand the smallest cardinality $|\mathcal B|$ of a base $\mathcal B\subset\U$.

A premetric $d$ on a quasi-uniform space $(X,\U)$ is called {\em $\U$-uniform} if for every $\e>0$ the set$$[d]_{<\e}=\{(x,y)\in X\times X:d(x,y)<\e\}$$belongs to the quasi-uniformity $\U$.

\begin{proposition}\label{p1.3} Each $\U$-uniform quasi-pseudometric $d$ on a quasi-uniform space $(X,\U)$ has open balls (with respect to the topology generated by the quasi-uniformity $\U$).
\end{proposition}

\begin{proof} Given a point $x\in X$ and a positive real number $\e>0$, we need to prove that the $\e$-ball $B_d(x,\e)$ is open in $X$. We need to check that each point $y\in B_d(x,\e)$ is an interior point of the ball $B_d(x,\e)$. Since the pseudo-quasimetric $d$ is $\U$-uniform, for the positive real number $\delta=\e-d(x,y)$ the ball $B_d(y,\delta)$ is a neighborhood of $y$ in $X$. The triangle inequality implies that $B_d(y,\delta)\subset B_d(x,\e)$, which means that $y$ is an interior point of $B_d(x,\e)$.
\end{proof}

\subsection{Rotund quasi-uniform spaces}

In this section we discuss rotund quasi-uniform spaces and some their modifications.

A base $\mathcal B$ of a quasi-uniformity $\U$ on a topological space $X$ is called
\begin{itemize}
\item {\em point-rotund}  $\overline{B(x;V)}\subset\overline{B(x;VU)}^\circ$ for any point $x\in X$ and entourages $V,U\in\mathcal B$;
\item {\em set-rotund} if $\overline{A}\subset \overline{B(A;U)}^\circ$ for any subset $A\subset X$ and entourage $U\in\mathcal B$;
\item {\em $\Delta$-rotund\/} if $B(\overline{B(x,V)};U)\subset \overline{B(x;VWU)}$ for any point $x\in X$ and entourages $U,V,W\in\mathcal B$;
\item {\em rotund\/} if $B(\overline{A};U)\subset \overline{B(A;WU)}$ for any subset $A\subset X$ and entourages $U,W\in\mathcal B$.
\end{itemize}
Here the closures and interiors are taken in the topology $\tau_\U$ generated by the quasi-uniformity $\mathcal U$.

For any base $\mathcal B\subset\U$ these notions relate as follows:
$$\xymatrix{\mbox{rotund}\ar@{=>}[r]\ar@{=>}[d]&\mbox{$\Delta$-rotund}\ar@{=>}[d]\\
\mbox{set-rotund}\ar@{=>}[r]&\mbox{point-rotund}.
}$$

A quasi-uniform space $(X,\U)$ is called {\em rotund} (resp. {\em point-rotund}, {\em set-rotund}, {\em $\Delta$-rotund}) if its quasi-uniformity $\U$ has a rotund (resp. point-rotund, set-rotund, $\Delta$-rotund) multiplicative base $\mathcal B\subset\U$.

%\begin{remark} Point-rotund and set-rotund quasi-uniform spaces were introduced in \cite{BR}. The notions of a ($\Delta$-) rotund quasi-uniform space seem to be new.
%\end{remark}

\begin{proposition}\label{p:uT} Each uniform space $(X,\U)$ is rotund.
\end{proposition}

\begin{proof} We claim that the multiplicative base $\mathcal B=\{U\in\U:U=U^{-1}\}$ of $\U$ is rotund. Given a subset $A\subset X$ and entourages $U,W\in \mathcal B$ we need to prove that $B(\overline{A};U)\subset \overline{B(A;WU)}$. Given any point $y\in B(\overline{A};U)$, we can find a point $z\in \overline{A}$ such that $y\in B(z;U)$.
It follows that $B(z;W^{-1})\cap A\ne\emptyset$ and hence $z\in B(A;W)$ and $y\in B(z;U)\subset B(B(A;W);U)=B(A;WU)\subset\overline{B(A;WU)}$.
\end{proof}

\subsection{Binary fractions} By $\IQ_2=\{\frac k{2^n}:k,n\in\IN,\;0<k<2^n\}$ we shall denote the set of binary fractions in the interval $(0,1)$ and by $\IQ_2^1=\IQ_2\cup\{1\}$ the set $\IQ_2$ enlarged by the unit.

Observe that each rational number $r\in \IQ_2$ can be uniquely written as the sum $r=\sum_{n=1}^{\infty}\frac{r_n}{2^n}$ for some binary sequence $\vec r=(r_n)_{n=1}^\infty\in\{0,1\}^{\IN}$ containing finitely many units and called
the {\em binary expansion} of $r$. Since $r>0$, the number $l_r=\max\{n\in\IN:r_n=1\}$ is well-defined. So, $r=\sum_{i=1}^{l_r}\frac{r_i}{2^i}$.
The binary expansion $\vec {\,r}=(r_n)_{n=1}^\infty\in\{0,1\}^\IN$ is uniquely determined by the finite set $\vec{\;r}^{-1}(1)=\{n\in\IN:r_n=1\}$ which can be enumerated as $\vec{\;r}^{-1}(1)=\{{\upa}r_1,\dots,{\upa}r_\ell\}$ for some increasing number sequence ${\upa r}_1<{\upa r}_2<\dots<{\upa r}_\ell$ of length $\ell=|\vec{\;r}^{-1}(1)|$. The $\ell$-tuple $({\upa r}_1,\dots,{\upa r}_\ell)$ will be denoted by $\upa r$ and called the {\em sequential expansion} of $r$. By $|\upa r|$ we shall denote the length $\ell$ of the sequential expansion $\upa r=({\upa r}_1,\dots,{\upa r}_\ell)$.
The rational number $r$ can be recovered from its sequential expansion $\upa r$ by the formula $r=\sum_{i=1}^\ell2^{-{\upa}r_i}$.

For two number sequences $a=(a_1,\dots,a_n)$ and $b=(b_1,\dots,b_m)$ by $a^\frown b=(a_1,\dots,a_n,b_1,\dots,b_m)$ we denote their {\em concatenation}. So, $\upa r=({\upa r}_1)^\frown ({\upa r}_2,\dots,{\upa r}_n)$ for any $r\in\IQ_2$.

\section{Right-continuous quasi-pseudometrics on rotund quasi-uniform spaces}\label{s2}

In this section we prove the principal (and technically the most difficult) result of this paper.

\begin{theorem}\label{main} Let $\U$ be a quasi-uniformity on a topological space $X$ and $\mathcal B$ be a multiplicative base for $\U$.
For any entourage $U\in\U$ there is an indexed family of entourages $\{V_q\}_{q\in\IQ^1_2}\subset\mathcal B$ such that
\begin{enumerate}
\item $V_q\circ V_r\subset V_{q+r}\subset U\subset V_1=X\times X$ for any rational numbers $q,r\in\IQ_2$ with $q+r<1$.
\item If the base $\mathcal B$ is point-rotund, then $\overline{V}_r\subset\overline{V}_q^\circ$ for any rational numbers $r,q\in\IQ_2$ with $r<q$.
\item If the base $\mathcal B$ is set-rotund, then $B(\overline{A},V_r)\subset \overline{B(A,V_q)}^\circ$ for any non-empty subset $A\subset X$ and rational numbers $r,q\in\IQ_2$ with $r<q$.
\item If the base $\mathcal B$ is $\Delta$-rotund, then
$\overline{V_p\kern-2pt}^\circ\circ\overline{V_q\kern-2pt}^\circ\subset\overline{V_r\kern-2pt}^\circ\subset \overline{U}^\circ$ for any rational numbers $p,q,r\in\IQ_2$ with $p+q<r$.
\end{enumerate}

The premetric $$d:X\times X\to[0,1],\;\;d:(x,y)\mapsto\inf\{r\in\IQ_2^1:(x,y)\in V_r\},$$ and its modifications $\overline{d}$ and $\overline{d}^\circ$
have the following properties:
\begin{enumerate}
\item[(5)] $d$ is a $\U$-uniform quasi-pseudometric with open balls such that $[d\/]_{<1}\subset U$;
\item[(6)] $\overline{d}\le\overline{d}^\circ\le d$;
\item[(7)] $\overline{d}$ is a $\U$-uniform premetric with closed balls such that $[\overline{d}]_{<1}\subset\overline{U}$ and\newline $\overline{d}(x,y)=\inf\{r\in\IQ_2^1:(x,y)\in \overline{V_r\kern-1pt}\,\}$ for every $x,y\in X$;
\item[(8)] $\overline{d}^\circ$ is a $\U$-uniform premetric with open balls such that $[\overline{d}^\circ]_{<1}\subset\overline{U}^\circ$ and\newline $\overline{d}^\circ(x,y)=\inf\{r\in\IQ_2^1:(x,y)\in \overline{V_r\kern-2pt}^\circ\}$  for every $x,y\in X$;
\smallskip
\item[(9)] If the base $\mathcal B$ is point-rotund, then the premetric $\overline{d}=\overline{d}^\circ$ is right-continuous.
\item[(10)] If the base $\mathcal B$ is set-rotund, then the premetric $d$ is $\overline{\dist}$-continuous and the premetric  $\overline{d}=\overline{d}^\circ$ is right-continuous and $\overline{\dist}$-continuous.
\item[(11)] If the base $\mathcal B$ is $\Delta$-rotund, then $\overline{d}=\overline{d}^\circ$ is a  quasi-pseudometric on $X$.
\item[(12)] If $B^{-1}=B$ for any $B\in\mathcal B$, then $d=\overline{d}^\circ=\overline{d}$ is a continuous pseudometric on $X$.
\end{enumerate}
\end{theorem}

\begin{proof} Given any entourage $U\in\U$, choose a sequence of entourages $(U_n)_{n=1}^\infty\in\mathcal B^\IN$ such that $U_1^3\subset U$ and $U_{n+1}^6\subset U_{n}$ for all $n\ge1$.

Put $V_1=X\times X$ and for every rational number $r\in\IQ_2$ with sequential expansion $\upa r=({\upa r}_1,\dots,{\upa r}_\ell)$ consider the sequence of entourages
$$U_{{\upa r}}:=(U_{{\upa r}_1},\dots,U_{{\upa r}_\ell})\in\mathcal B^\ell$$and their balanced product
$$V_r:=\Pi(U_{\upa r})=\Pi_\ell(U_{{\upa}r_1},\dots,U_{{\upa}r_\ell}),$$
which belongs to the family $\mathcal B$ by the multiplicativity of $\mathcal B$.
In Claims~\ref{cl2.3}--\ref{cl2.5} we shall prove that the family $(V_r)_{r\in\IQ_2^1}$ satisfies the conditions (1)--(4) of the theorem.

\begin{claim}\label{cl2.2} For any rational number $x\in\IQ_2$ with sequential expansion $\upa x$ we get
$V_x\subset U_{{\upa x}_1}^3\subset U.$
\end{claim}

\begin{proof} This claim will be proved by induction on the length $\ell=|\upa x|$ of the sequence $\upa x=({\upa x}_1,\dots,{\upa x}_\ell)$. If $\ell=1$, then $V_x=U_{{\upa x}_1}\subset U_{{\upa x}_1}^3$. Assume that for some integer number $\ell>1$ the inclusion $V_x\subset U_{{\upa x}_1}^3$ has been proved for all rational numbers $x\in\IQ_2$ with $|\upa x|<\ell$. Take any rational number $x\in\IQ_2$ with $|\upa x|=\ell$. Consider the rational number $\tilde x=x-2^{-{\upa x}_1}$ and observe that the sequence ${\upa \tilde x}=({\upa x}_2,\dots,{\upa x}_\ell)$ has length $|\upa\tilde x|=\ell-1$. Then by the induction hypothesis, $V_{\tilde x}\subset U_{{\upa\tilde x}_1}^3=U_{{\upa x}_2}^3$. Since $\upa x=({\upa x}_1)^\frown \upa \tilde x$, we get $$V_x=\Pi(U_{\upa x})=\Pi(U_{{\upa x}_1},U_{\upa\tilde x})=\Pi(U_{\upa\tilde x})U_{{\upa x}_1}\Pi(U_{\upa\tilde x})=V_{\tilde x}U_{{\upa x}_1}V_{\tilde x}\subset U_{{\upa x}_2}^3U_{{\upa x}_1}U_{{\upa x}_2}^3\subset U_{{\upa x}_2-1}U_{{\upa x}_1}U_{{\upa x}_2-1}\subset U_{{\upa x}_1}^3.$$
\end{proof}

\begin{claim}\label{cl2.3}  For any rational numbers $x,y\in\IQ_2$ with $x+y<1$ we get $V_xV_y\subset V_{x+y}$.
\end{claim}

\begin{proof} Let $\upa x$, $\upa y$ and $\upa z$ be the sequential expansions of the rational numbers $x$, $y$, and $z=x+y$, respectively. The inclusion $V_x\circ V_y\subset V_{x+y}$ will be proved by induction on $n=|\upa x|+|\upa y|\ge 2$. To start the induction, assume that $n=2$. Then $\upa x=({\upa x}_1)$ and $\upa y=(\upa y_1)$ and either ${\upa x}_1=\upa y_1$ or ${\upa x}_1<\upa y_1$ or $\upa y_1<{\upa x}_1$. If ${\upa x}_1=\upa y_1$, then $\upa z=({\upa x}_1-1)$. In this case
$$V_xV_y=U_{{\upa x}_1}U_{\upa y_1}=U_{{\upa x}_1} U_{{\upa x}_1}\subset U_{{\upa x}_1-1}=U_{\upa z_1}=V_z.$$
If ${\upa x}_1<\upa y_1$, then $\upa z=({\upa x}_1,\upa y_1)$ and hence
$$V_xV_y=U_{{\upa x}_1}U_{\upa y_1}\subset U_{\upa y_1}U_{{\upa x}_1}U_{\upa y_1}=\Pi(U_{{\upa x}_1},U_{\upa y_1})=\Pi(U_{\upa z_1},U_{\upa z_2})=V_z.$$
If $\upa y_1<{\upa x}_1$, then $\upa z=(\upa y_1,{\upa x}_1)$ and hence
$$V_xV_y=U_{{\upa x}_1}U_{\upa y_1}\subset U_{{\upa x}_1}U_{\upa y_1}U_{{\upa x}_1}=\Pi(U_{\upa y_1},U_{{\upa x}_1})=\Pi(U_{\upa z_1},U_{\upa z_2})=V_z.$$

Now assume that for some $n\ge 2$ and all sequences $x,y\in\IQ_2$ with $z=x+y<1$ and $|\upa x|+|\upa y|\le n$ the inclusion $V_xV_y\subset V_z$ has been established. Take any sequences $x,y\in\IQ_2$ with $z=x+y<1$ and $|\upa x|+|\upa y|=n+1$. Let $\upa x$, $\upa y$, and $\upa z$ be the sequential expansions of the rational numbers $x$, $y$, and $z$, respectively.
Two cases are possible.

1: $\upa z_1<\min\{{\upa x}_1,\upa y_1\}$. In this case Claim~\ref{cl2.2} implies
$$V_xV_y\subset U_{{\upa x}_1}^3U_{\upa y_1}^3\subset U_{\upa z_1+1}^3U_{\upa z_1+1}^3=U_{\upa z_1+1}^6\subset U_{\upa z_1}\subset V_z.$$

2: ${\upa z}_1=\min\{{\upa x}_1,{\upa y}_1\}$. In this case $\upa x_1\ne\upa y_1$ and hence $\upa z_1=\upa x_1<\upa y_1$ or $\upa z_1=\upa y_1<\upa x_1$.

If $\upa z_1=\upa x_1<\upa y_1$, then consider the rational numbers $\tilde x=x-2^{-\upa x_1}$ and $\tilde z=z-2^{-\upa z_1}$. Then $\upa x=(\upa x_1)^\frown \upa \tilde x$ and $\upa z=(\upa z_1)^\frown \upa\tilde z$. Since $|\upa \tilde x|+|\upa y|=(n+1)-1=n$, by the inductive assumption, $V_{\tilde x}\subset V_{\tilde x}V_y\subset V_{\tilde x+y}=V_{\tilde z}.$ Then
$$V_xV_y=V_{\tilde x}U_{\upa x_1}V_{\tilde x}V_y\subset V_{\tilde z}U_{{\upa\tilde z}_1}V_{\tilde z}=V_z.$$

If $\upa z_1=\upa y_1<\upa x_1$, then consider the rational numbers $\tilde y=y-2^{-\upa y_1}$ and $\tilde z=z-2^{-\upa z_1}$. It follows that $\upa y=(\upa y_1)^\frown \upa \tilde y$ and $\upa z=\upa z_1^\frown \upa\tilde z$. Since $|\upa x|+|\upa \tilde y|=n$,  the inductive assumption guarantees that $V_{\tilde y}\subset V_xV_{\tilde y}\subset V_{x+\tilde y}=V_{\tilde z}$ and thus
$$V_xV_y=V_xV_{\tilde y}U_{\upa y_1}V_{\tilde y}\subset V_{\tilde z}U_{{\upa\tilde z}_1}V_{\tilde z}=V_z.$$
\end{proof}

\begin{claim}\label{cl2.4}  If the base $\mathcal B$ is point-rotund, then $\overline{B}(x,V_r)\subset \overline{B}^\circ(x,V_q)$ for any point $x\in X$ and rational numbers $r,q\in\IQ_2$ with $r<q$.
\end{claim}

\begin{proof}
Choose any number $n\in\IN$ such that  $r+2^{-n}<q$. Since the base $\mathcal B$ is point-rotund, for the entourages $V_r,U_n\in\mathcal B$ we get
$$\overline{B}(x;V_r)\subset
\overline{B}^\circ(x;V_rU_n)=\overline{B}^\circ(x;V_rV_{2^{-n}})\subset \overline{B}^\circ(x;V_{r+2^{-n}})\subset \overline{B}^\circ(x;V_q).
$$
\end{proof}

By analogy we can prove:

\begin{claim}\label{cl2.5}  If the base $\mathcal B$ is set-rotund, then $\overline{B}(A,V_r)\subset \overline{B}^\circ(A,V_q)$ for any subset $A\subset X$ and rational numbers $r,q\in\IQ_2$ with $r<q$.
\end{claim}

\begin{claim}\label{cl2.6}  If the family $\mathcal B$ is $\Delta$-rotund, then
$\overline{V}_p^\circ\circ\overline{V}_q^\circ\subset\overline{V}_p\circ\overline{V}_q\subset \overline{V}^\circ_{r}\subset \overline{U}^\circ$ for any rational numbers $p,q,r\in\IQ_2$ with $p+q<r$.
\end{claim}

\begin{proof} Given any points $x\in X$, $y\in  \overline{B}(x;V_p)$ and $z\in  \overline{B}(y;V_q)$, we need to prove $z\in  \overline{B}^\circ(x;V_{r})$. Choose any number $n\in\IN$ such that $p+q+2^{-n}<r$. Taking into account that $\mathcal B$ is $\Delta$-rotund and applying Claims~\ref{cl2.3} and \ref{cl2.4}, we conclude that
$$
z\in \overline{B}(y;V_q)\subset
\overline{B(\overline{B}(x;V_p);V_q)}\subset \overline{B(x;V_pU_nV_q)}=\overline{B}(x;V_pV_{2^{-n}}V_q)\subset
\overline{B}(x;V_{p+2^{-n}+q})\subset
\overline{B}^\circ(x;V_r).
$$
\end{proof}

Now consider the premetric $$d:X\times X\to[0,1],\;\;d(x,y)=\inf\{r\in\IQ_2^1:(x,y)\in V_r\}$$ and its modifications $\overline{d}$ and $\overline{d}^\circ$.
In the following claims we shall prove that they satisfy the conditions (5)--(12) of the theorem. The condition (6) follows from the definition of the modifications $\overline{d}$ and $\overline{d}^\circ$.

\begin{claim}\label{cl2.7}  The premetric $d$ is a $\U$-uniform quasi-pseudometric with open balls, and $[d]_{<1}\subset U$.
\end{claim}

\begin{proof} To prove that $d$ is $\U$-uniform, we need to check that for every $\e>0$ the set $[d]_{<\e}=\{(x,y)\in X\times X:d(x,y)<\e\}$ belongs to the quasi-uniformity $\U$. Take any positive integer $m\in\IN$ with $2^{-m}<\e$ and observe that for every $x\in X$ and $y\in B(x,V_{2^{-m}})$ we get $d(x,y)\le 2^{-m}<\e$ and hence
$\U\ni U_m=V_{2^{-m}}\subset \{(x,y)\in X\times X:d(x,y)<\e\}=[d]_{<\e}$, which implies $[d]_{<\e}\in\U$.
\smallskip

To show that $d$ is a quasi-pseudometric, it suffices to show that $d(x,z)\le d(x,y)+d(y,z)+2\e$ for every points $x,y,z\in X$ and every $\e>0$. This inequality trivially holds if $d(x,y)+d(y,z)+2\e\ge 1$. So, we assume that $d(x,y)+d(y,z)+2\e<1$. Choose any rational numbers $r,q\in\IQ_2$ such that $d(x,y)<r<d(x,y)+\e$ and $d(y,z)<q<d(y,z)+\e$. By the definition of the premetric $d$ we get
$(x,y)\in V_r$ and $(y,z)\in V_q$ and hence $(x,z)\in V_rV_q\subset V_{r+q}$ according to Claim~\ref{cl2.3}. It follows that $d(x,z)\le r+q<d(x,y)+d(y,z)+2\e$.

By Proposition~\ref{p1.3}, the $\U$-uniform quasi-pseudometric $d$ has open balls.
Finally, observe that $[d]_{<1}\subset \bigcup_{r\in\IQ_2}V_r\subset U$.
\end{proof}

\begin{claim}\label{cl2.8}
$\overline{d}$ is a $\U$-uniform premetric with closed balls such that $[\overline{d}]_{<1}\subset\overline{U}$ and  $$\overline{d}(x,y)=\inf\{r\in\IQ_2^1:(x,y)\in \overline{V_r\kern-1pt}\,\}\mbox{ \  for every $x,y\in X$}.$$
\end{claim}

\begin{proof} The inequality $\overline{d}\le d$ implies that $[d]_{<\e}\subset[\overline{d}]_{<\e}$ for every $\e>0$. By Claim~\ref{cl2.7}, $[d]_{<\e}\in\U$ and hence $[\overline{d}]_{<\e}\in\U$, which means that $\overline{d}$ is $\U$-uniform. By Proposition~\ref{p1.4n}(1), $\overline{d}$ is a premetric with closed balls. The inclusion $[\overline{d}]_{<1}\subset\overline{U}$ follows from the inclusion $[d]_{<1}\subset U$ and the definition of the regularization $\overline{d}$.

Finally, we prove that  $\overline{d}(x,y)=\inf\{r\in\IQ_2^1:y\in \overline{B(x,V_r)}\}$ for any points $x,y\in X$. Let $\rho(x,y)=\inf\{r\in \IQ_2^1:y\in\overline{B(x,V_r)}\}$.

First we show that $\overline{d}(x,y)\ge \rho(x,y)$. Assuming that $\overline{d}(x,y)<\rho(x,y)$, we can find a rational number $r\in\IQ_2$ such that $r<\rho(x,y)$ and $y\in\overline{B_d(x,r)}$. Since $B_d(x,r)\subset B(x;V_r)$, we conclude that $y\in\overline{B_d(x,r)}\subset \overline{B(x;V_r)}$ and hence $\rho(x,y)\le r<\rho(x,y)$, which is a desired contradiction.

To show that $\overline{d}(x,y)\le \rho(x,y)$, it suffices to check that $\overline{d}(x,y)\le \delta$ for any real number $1\ge \delta>\rho(x,y)$. By the definition of $\rho(x,y)$, there is a rational number $r\in\IQ_2$ such that $r<\delta$ and $y\in \overline{B(x;V_r)}$. The definition of the quasi-pseudometric $d$ guarantees that $B(x;V_r)\subset B_d(x,\delta)$. Then $y\in \overline{B(x;V_r)}\subset\overline{B_d(x,\delta)}$ and hence $\overline{d}(x,y)\le\delta$.
\end{proof}

By analogy we can prove:

\begin{claim}\label{cl2.9}
 $\overline{d}^\circ$ is a $\U$-uniform premetric with open balls such that $[\overline{d}^\circ]_{<1}\subset\overline{U}^\circ$ and
 $$\overline{d}^\circ(x,y)=\inf\{r\in\IQ_2^1:(x,y)\in \overline{V_r\kern-2pt}^\circ\}\mbox{ \  for every $x,y\in X$}.
 $$
\end{claim}

\begin{claim}\label{cl2.10}  If the base $\mathcal B$ is point-rotund,  then the premetric $\overline{d}=\overline{d}^\circ$ is right-continuous.
\end{claim}

\begin{proof} Claims~\ref{cl2.4}, \ref{cl2.8} and \ref{cl2.9} imply that $\overline{d}=\overline{d}^\circ$. By Corollary~\ref{c1.6n}, the premetric $\overline{d}=\overline{d}^\circ$ is right-continuous.
\end{proof}

\begin{claim}\label{cl2.11} If the base $\mathcal B$ is set-rotund, then the premetrics $d$ and $\overline{d}=\overline{d}^\circ$ are $\overline{\dist}$-continuous.
\end{claim}

\begin{proof} Assuming that $\mathcal B$ is set-rotund, we shall show that $\overline{d}_A=\overline{d}^\circ_A$ for any non-empty subset $A\subset X$.
Since $\overline{d}_A\le \overline{d}^\circ_A$, it suffices to check that $\overline{d}_A(x)\ge \overline{d}^\circ_A(x)$ for every $x\in X$. This inequality will follow as soon as we check that
$\overline{d}^\circ_A(x)< \overline{d}_A(x)+\e$ for every $\e>0$. This inequality is trivial if
$\overline{d}_A(x)+\e\ge 1$. So we assume that  $\overline{d}_A(x)+\e<1$. Choose  two rational numbers $r,q\in\IQ_2$ such that $\overline{d}_A(x)<r<r+2q<\overline{d}_A(x)+\e$. The definition of $\overline{d}_A(x)$ and Claim~\ref{cl2.3} imply that
$$
x\in \overline{B_{d}(A,r)}\subset
\overline{B(A;V_r)}\subset \overline{B(B(A;V_r);V_q)}^\circ=
\overline{B(A;V_rV_q)}^\circ\subset \overline{B(A;V_{r+q})}^\circ\subset \overline{B_d(A,r+2q)}^\circ .$$
Then the definition  of the distance function $\overline{d}^\circ_A$ yields the desired inequality
$\overline{d}^\circ_A(x)\le r+2q<\overline{d}_A(x)+\e$. So, $\overline{d}_A=\overline{d}^\circ_A$.
By Proposition~\ref{p1.2n}, the function $\overline{d}_A=\overline{d}^\circ_A:X\to[0,1]$ is continuous. This means that the premetric $d$ is $\overline{\dist}$-continuous. By Corollary~\ref{c1.6n}, the premetric $\overline{d}=\overline{d}^\circ$ is $\overline{\dist}$-continuous.
\end{proof}

\begin{claim}\label{cl2.12}  If the base $\mathcal B$ is $\Delta$-rotund, then
$\overline{d}=\overline{d}^\circ$ is a quasi-pseudometric.
\end{claim}

\begin{proof} Assuming that the base $\mathcal B$ is $\Delta$-rotund, we conclude that $\mathcal B$ is point-rotund. By Claim~\ref{cl2.10}, $\overline{d}=\overline{d}^\circ$.

To show that $\overline{d}=\overline{d}^\circ$ is a quasi-pseudometric, it suffices to show that $\overline{d}(x,z)\le \overline{d}(x,y)+\overline{d}(y,z)+\e$ for every points $x,y,z\in X$ and every $\e>0$. This inequality trivially holds if $\overline{d}(x,y)+\overline{d}(y,z)+\e\ge 1$. So, we assume that $\overline{d}(x,y)+\overline{d}(y,z)+\e<1$. Choose any rational numbers $p,q,r\in\IQ_2$ such that $\overline{d}(x,y)<p$ and $\overline{d}(y,z)<q$ and $p+q<r<\overline{d}(x,y)+\overline{d}(y,z)+\e$. By Claim~\ref{cl2.8}, $(x,y)\in \overline{V}_p$ and $(y,z)\in \overline{V}_q$ and hence $(x,z)\in \overline{V}_p\overline{V}_q\subset \overline{V}_{r}$ according to Claim~\ref{cl2.6}. By Claim~\ref{cl2.8}, $\overline{d}(x,z)\le r<\overline{d}(x,y)+\overline{d}(y,z)+\e$.
\end{proof}

\begin{claim}\label{cl2.13}  If $B^{-1}=B$ for all $B\in\mathcal B$, then $d=\overline{d}^\circ=\overline{d}$ is a continuous pseudometric on $X$.
\end{claim}

\begin{proof} It follows that $V_r^{-1}=V_r$ for all $r\in\IQ_2$ and hence $d(x,y)=d(y,x)$ for all $x,y\in X$ by the definition of the premetric $d$. Being a symmetric quasi-pseudometric, the premetric $d$ is a pseudometric. Since the pseudometric $d$ has open balls, it is continuous. For every rational numbers $p,q\in\IQ_2$ with $p<q$ we get
$$\overline{B}(x;V_p)\subset \overline{B}(x;[d]_{\le p})=B(x;[d]_{\le p})\subset B(x;[d]_{<q})\subset B(x;V_q),$$which implies the inequality $d\le\overline{d}$. Since $\overline{d}\le \overline{d}^\circ\le d$, this inequality yields the equality $\overline{d}=\overline{d}^\circ=d$.
\end{proof}
\end{proof}

Theorem~\ref{main} has the following immediate corollary.

\begin{corollary}\label{c:u-main} For every $\Delta$-rotund (quasi-)uniform space $(X,\U)$ and every entourage $U\in\U$ there is a right-continuous $\U$-uniform (quasi-)pseudometric $p:X\times X\to[0,1]$ such that $B_p(x,1)\subset \overline{B}^\circ(x;U)$ for every $x\in X$.
\end{corollary}

\section{Metrizability theorems for quasi-uniform spaces}\label{s3}

In this section we shall apply Theorem~\ref{main} and prove some metrizability theorems for quasi-uniform spaces.

We say that a quasi-uniformity $\U$ on a set $X$ is generated by a family of premetrics $\mathcal D$ on $X$ if the family $$\{[d]_{<\e}:d\in\mathcal D,\;\e>0\}$$ is a subbase of the quasi-uniformity $\U$.

It is easy to see that each family of pseudometrics $\mathcal D$ on a set $X$ generates a unique quasi-uniformity on $X$ (namely, that generated by the subbase $\{[d]_{<\e}:d\in\mathcal D,\;\e>0\}$).

Theorem~\ref{main} gives an alternative (and direct) proof of the following classical result (see  Kelley \cite[6.12--6.14]{Kelley}).

\begin{theorem}\label{t3.1} Each quasi-uniformity $\U$ is generated by a family $\mathcal D$ of quasi-pseudometrics, which has cardinality $|\mathcal D|\le\chi(\U)$.
\end{theorem}

\begin{proof} Fix a base $\mathcal B$ of the quasi-uniformity $\U$ of cardinality $|\mathcal B|=\chi(\U)$. By Theorem~\ref{main}, for every entourage $U\in\mathcal B$ there exists a $\U$-uniform quasi-pseudometric $dU$ such that $[dU]_{<1}\subset U$. Then $\mathcal D=\{dU:U\in\mathcal B\}$ is a family of quasi-pseudometrics generating the quasi-uniformity $\U$ and having cardinality $|\mathcal D|\le |\mathcal B|=\chi(\U)$.
\end{proof}

\begin{corollary}\label{c3.2} A quasi-uniformity $\U$ is generated by a single quasi-pseudometric if and only if $\U$ has countable uniform weight $\chi(\U)$.
\end{corollary}

\begin{proof} If the quasi-uniformity $\U$ is generated by a quasi-pseudometric $d$, then the family $\{[d]_{<\frac1n}:n\in\IN\}$ is a countable base of the quasi-uniformity $\U$, which implies that $\chi(\U)$ is countable.

Assuming conversely that $\chi(\U)$ is countable and applying Theorem~\ref{t3.1}, we can find a countable family of quasi-pseudometrics $\mathcal D=\{d_n\}_{n\in\IN}$ generating the quasi-uniformity $\U$.
It is easy to check that the quasi-pseudometric $d=\max_{n\in\IN}\min\{d_n,2^{-n}\}$  generates the quasi-uniformity $\U$.
\end{proof}

Next, we deduce from Theorem~\ref{main} even more classical result whose history traces back to Aleksandrov, Urysohn \cite{AU}, Chittenden \cite{Chit}, Frink \cite{Frink}, Aronszajn \cite{Aron} and finally Weil \cite{Weil} (more historical information can be found in Kelley \cite[6.14]{Kelley} and  Engelking \cite[p.436]{Eng}).

\begin{theorem}\label{t3.3}  Each uniformity $\U$ is generated by a family $\mathcal D$ of (continuous) pseudometrics of cardinality $|\mathcal D|\le\chi(\U)$.
\end{theorem}

\begin{proof} Fix a base $\mathcal V$ of the quasi-uniformity $\U$ of cardinality $|\mathcal V|=\chi(\U)$. It follows that the family $\mathcal B=\{U\in\U:U=U^{-1}\}$ is a multiplicative base of the quasi-uniformity $\U$.

By Theorem~\ref{main}(9), for every entourage $U\in\mathcal B$ there exists a $\U$-uniform continuous pseudometric $dU$ such that $[dU]_{<1}\subset U$. Then $\mathcal D=\{dU:U\in\mathcal B\}$ is a family of continuous pseudometrics generating the uniformity $\U$ and having cardinality $|\mathcal D|\le |\mathcal B|=\chi(\U)$.
\end{proof}

Now we prove a new result characterizing uniformly regular quasi-uniform spaces among (point-)rotund quasi-uniform spaces.

A quasi-uniform space $(X,\U)$ is called
\begin{itemize}
\item {\em uniformly regular} if for any $U\in\U$ there is $V\in\U$ such that $\overline{B}(x;V)\subset B(x;U)$ for every $x\in X$;
\item {\em uniformly semiregular} if for any $U\in\U$ there is $V\in\V$ such that $\overline{B}^\circ\kern-2pt(x;V)\subset B(x;U)$ for every $x\in X$;
\item {\em uniformly completely regular} if for any entourage $U\in\U$ there exists a $\U$-uniform right continuous premetric $d:X\times X\to[0,1]$ such that $B_d(x,1)\subset B(x;U)$ for every $x\in X$.
\end{itemize}
For each quasi-uniform space we have the implications:
$$\xymatrix{
\mbox{uniformly completely regular}\ar@{=>}[r]&\mbox{uniformly regular}\ar@{=>}[r]&\mbox{uniformly semiregular}.}
$$

\begin{theorem}\label{t3.4} For a point-rotund quasi-uniform space $(X,\U)$ the following conditions are equivalent:
\begin{enumerate}
\item $(X,\U)$ is uniformly semiregular;
\item $(X,\U)$ is uniformly regular;
\item $(X,\U)$ is uniformly completely regular.
\end{enumerate}
If the uniform space $(X,\U)$ is rotund, then the conditions \textup{(1)--(3)} are equivalent to:
\begin{enumerate}
\item[(4)] the quasi-uniformity $\U$ is generated by a family of right-continuous premetrics;
\item[(5)] the quasi-uniformity $\U$ is generated by a family $\mathcal D$ of right-continuous and $\overline{\dist}$-continuous quasi-pseudo\-metrics having cardinality $|\mathcal D|\le\chi(\U)$.
\end{enumerate}
\end{theorem}

\begin{proof} The implications $(3)\Ra(2)\Ra(1)$ are trivial. To prove that $(1)\Ra(3)$, assume that the quasi-uniform space $(X,\U)$ is uniformly semiregular. To prove that $(X,\U)$ is uniformly completely regular, fix any entourage $U\in\U$. By the uniform semiregularity, the quasi-uniformity $\U$ contains an entourage $V\in\U$ such that $\overline{V}^\circ\subset U$. By Theorem~\ref{main}(7), there exists a $\U$-uniform right-continuous premetric $p:X\times X\to[0,1]$ such that $B_p(x,1)\subset B(x,\overline{V}^\circ)=\overline{B}^\circ(x,V)\subset B(x,U)$ for all $x\in X$. This witnesses that the quasi-uniform space $(X,\U)$ is uniformly completely regular.

Now assuming that the quasi-uniform space $(X,\U)$ is rotund, we shall prove that the conditions (1)--(3) are equivalent to the conditions (4) and (5).

First we prove that $(1)\Ra(5)$. Fix a base $\mathcal B$ of the quasi-uniformity $\U$ of cardinality $|\mathcal B|=\chi(\U)$. Since the quasi-uniform space $(X,\U)$ is uniformly semiregular, for every entourage $U\in\mathcal B$ there is an entourage $V\in\U$ such that $\overline{V}^\circ\subset U$. By Theorem~\ref{main}(9,10), there exists a right-continuous and $\overline{\dist}$-continuous  quasi-pseudometric $dU:X\times X\to[0,1]$ such that $[dU]_{<1}\subset\overline{V}^\circ \subset U$. Then the family $\mathcal D=\{dU:U\in\mathcal B\}$ of right-continuous and $\overline{\dist}$-continuous quasi-pseudometrics generates the uniformity $\U$ and has cardinality $|\mathcal D|\le|\mathcal B|=\chi(\U)$.

The implication $(5)\Ra(4)$ is trivial. To prove $(4)\Ra(3)$, assume that the quasi-uniformity $\U$ is generated by a family $\mathcal D$ of right-continuous premetrics. We lose no generality assuming that $\mathcal D$ is the family of all $\U$-uniform right-continuous premetrics on $X$. In this case the family $\mathcal B=\{[d]_{<1}:d\in\mathcal D\}$ is a base of the quasi-uniformity $\U$. To show that the quasi-uniform space $(X,\U)$ is uniformly completely regular, fix any entourage $U\in\U$ and find a right-continuous premetric $d\in\mathcal D$ such that $[d]_{<1}\subset U$. The premetric $d$ witnesses that the quasi-uniform space $(X,\U)$ is uniformly completely regular.
\end{proof}

The following corollaries can be deduced from Theorem~\ref{t3.4}(5) by analogy with Corollary~\ref{c3.2}.

\begin{corollary}\label{c3.5} The quasi-uniformity of a rotund quasi-uniform space $(X,\U)$ is generated by a single right-continuous quasi-pseudometric if and only if $(X,\U)$ is semi-regular and $\chi(\U)\le\w$.
\end{corollary}

\begin{corollary}\label{c3.6} The quasi-uniformity of a rotund quasi-uniform space $(X,\U)$ is generated by a single right-continuous quasi-metric if and only if $(X,\U)$ is semi-regular, $\chi(\U)\le\w$ and $(X,\tau_\U)$ is a $T_1$-space.
\end{corollary}

\section{Point-rotund topological spaces}\label{s4}

A topological space $X$ is defined to be ({\em point-}){\em rotund} if the topology of $X$ is generated by a quasi-uniformity that has a (point-)rotund multiplicative base. The following corollary of Theorem~\ref{main} was first proved by the authors  \cite{BR} by a different method.

\begin{theorem}\label{t4.1} If a topological space $X$ is point-rotund, then for any point $x\in X$ and  neighborhood $O_x\subset X$ of $x$ there exists a continuous function $f:X\to[0,1]$ such that $f(x)=0$ and $f^{-1}\big([0,1)\big)\subset\overline{O}^\circ_x$.
\end{theorem}

\begin{proof} By our assumption, the topology of $X$ is generated by a  quasi-uniformity $\U$ having a  point-rotund multiplicative base $\mathcal B$. It this case for every point $x\in X$ and a neighborhood $O_x\subset X$ there exists an entourage $U\in\U$ such that $B(x;U)\subset O_x$. By Theorem~\ref{main}(7,9), there is a right-continuous premetric $p:X\times X\to[0,1]$ such that $B_p(x,1)\subset \overline{B}^\circ(x;U)\subset\overline{O}^\circ_x$. Then the function $f:X\to[0,1]$, $f:y\mapsto p(x,y)$, is continuous and has the desired properties $f(x)=0$ and  $f^{-1}\big([0,1)\big)\subset \overline{O}^\circ_x$.
\end{proof}

This theorem implies the following corollary first proved in \cite{BR}:

\begin{corollary}\label{c4.2} A point-rotund topological space $X$ is:
\begin{enumerate}
\item completely regular if and only if $X$ is regular if and only if $X$ is semiregular;
\item functionally Hausdorff if and only if $X$ is Hausdorff if and only if $X$ is semi-Hausdorff.
\end{enumerate}
\end{corollary}

This theorem shows that for point-rotund topological spaces the diagram describing the interplay between various separation properties transforms to the following symmetric form:
$$\xymatrix{
&&T_{\frac12 2}\ar@{<=>}[d]&T_{\frac12 3}\ar@{=>}[l]\ar@{<=>}[d]\ar@{=>}[r]&\tfrac12R\\
T_0&T_1\ar@{=>}[l]&T_2\ar@{=>}[l]&T_3\ar@{=>}[l]\ar@{=>}[r]&R\ar@{<=>}[u]\ar@{<=>}[d]&\hskip-55pt.\\
&&T_{2\frac12}\ar@{<=>}[u]&T_{3\frac12}\ar@{=>}[l]\ar@{<=>}[u]\ar@{=>}[r]&R\tfrac12\\
}
$$

Corollary~\ref{c4.2} combined with Proposition~\ref{p:uT} imply the following characterization of point-rotund semiregular spaces.

\begin{corollary} A semi-regular space $X$ is point-rotund if and only if $X$ is completely regular.
\end{corollary}

We do not known if Theorem~\ref{t4.1} can be reversed.

\begin{problem} Is a topological space $X$ point-rotund if for any point $x\in X$ and  neighborhood $O_x\subset X$ of $x$ there exists a continuous function $f:X\to[0,1]$ such that $f(x)=0$ and $f^{-1}\big([0,1)\big)\subset\overline{O}^\circ_x$?
\end{problem}

\section{Quasi-uniformities on topological semigroups and monoids}\label{s5}

In this section we study four natural quasi-uniformities on  topological
semigroups and topological monoids. Some of the obtained results are known (see \cite{Koper}, \cite{KMR}) and we include them for the convenience of the reader.

We recall that a {\em topological semigroup} is a topological space $S$ endowed with a continuous associative binary operation $S\times S\to S$, $(x,y)\mapsto xy$.

An element $e\in S$ of a topological semigroup $S$ is called
\begin{itemize}
\item a {\em left unit} if $ex=x$ for all $x\in S$;
\item a {\em right unit} if $xe=x$ for all $x\in S$;
\item a {\em unit} if $xe=x=ex$ for all $x\in S$.
\end{itemize}
It is well-known (and easy to prove) that any two units in a semigroup coincide. A  topological semigroup $S$ possessing a unit is called a {\em topological monoid}. In the sequel for a (left, right) unit $e$ in a topological semigroup $S$ by $\mathcal N_e$ we denote a neighborhood base at $e$.

Any topological semigroup $S$ with a right unit $e$ carries the {\em left quasi-uniformity} $\mathcal L$ generated by the base $\mathcal B_{\mathcal L}=\big\{\{(x,y)\in S\times S:y\in xU\}:U\in\mathcal N_e\}$, and any  topological semigroup $S$ with a left unit $e$ carries the {\em right quasi-uniformity} $\mathcal R$ generated by the base $\mathcal B_{\mathcal R}=\big\{\{(x,y)\in S\times S:y\in Ux\}:U\in\mathcal N_e\}$.

For any topological monoid $S$ both quasi-uniformities $\mathcal L$ and $\mathcal R$ are well-defined. These quasi-uniformities induce two other quasi-uniformities: $\mathcal L\vee\mathcal R$ and $\mathcal L\wedge\mathcal R$, which are generated by the bases:
$$
\begin{aligned}
\mathcal B_{\mathcal L\vee\mathcal R}&=\big\{\{(x,y)\in S\times S:y\in xU\cap Ux\}:U\in\mathcal N_e\big\}\mbox{ \ and}\\
\mathcal B_{\mathcal L\wedge\mathcal R}&=\big\{\{(x,y)\in S\times S:y\in UxV\}:U,V\in\mathcal N_e\big\},
\end{aligned}
$$
respectively.

Now we will detect topological monoids whose topology is compatible with the quasi-uniformities $\mathcal L$, $\mathcal R$, $\mathcal L\vee\mathcal R$ and $\mathcal L\wedge\mathcal R$.

We shall say that a topological semigroup $S$ has:
\begin{itemize}
\item {\em open left shifts} if for every $a\in S$ the left shift $\ell_a:S\to S$, $\ell_a:x\mapsto ax$, is an open map of $S$;
\item {\em open right shifts} if for every $a\in S$ the right shift $r_a:S\to S$, $r_a:x\mapsto xa$, is an open map of $S$;
\item {\em open shifts} if $S$ has open left shifts and open right shifts;
\item {\em open central shifts} if for every $a\in S$ the map $c_a:S\times S\to S$, $c_a:(x,y)\mapsto xay$, is open.
\end{itemize}

These notions are tightly connected with the openness of shifts at the unit.

An element $e\in S$ of a topological semigroup $S$ is called
\begin{itemize}
\item an {\em open left unit} if for every neighborhood $U\subset S$ of $e$ and every $x\in S$ the set $Ux=\{ux:u\in U\}$ is a neighborhood of the point $ex=x$ in $S$;
\item an {\em open right unit}  if for every neighborhood $U\subset S$ of $e$ and every $x\in S$ the set $xU=\{xu:u\in U\}$ is a neighborhood of the point $xe=x$ in $S$;
\item an {\em open  unit} if for every neighborhood $U\subset S$ of $e$ and every $x\in S$ the sets $UxU$ is a neighborhood of the point $x=ex=xe$ in $S$.
\end{itemize}

%\begin{remark}
%It is clear that each paratopological group is a topological semigroup with an open unit. The closed half-line $[0,\infty)$ endowed the Sorgenfrey topology (generated by the base $\mathcal B=\{[a,b):0\le a<b<\infty\}$) and the operation of addition of real numbers a topological semigroup with an open unit, which is not a paratopological group.
%\end{remark}

In the following four propositions we characterize topological monoids whose quasi-uniformities
$\mathcal L,\mathcal R, \mathcal L\vee\mathcal R$, and $\mathcal L\wedge\mathcal R$ generate the topology of the monoid.

\begin{proposition}\label{p5.1} For a topological monoid $S$ the following conditions are equivalent:
\begin{enumerate}
\item the quasi-uniformity $\mathcal L\wedge \mathcal R$ generates the topology of $S$;
\item the unit $e$ of $S$ is an open unit;
\item the semigroup $S$ has open central shifts.
\end{enumerate}
The conditions \textup{(1)--(3)} imply:
\begin{enumerate}
\item[(4)] the quasi-uniformity $\mathcal L\wedge \mathcal R$ is rotund.
\end{enumerate}
\end{proposition}

\begin{proof} The implications $(3)\Ra(1)\Leftrightarrow(2)$ follow easily from the corresponding definitions.
\smallskip

$(2)\Ra(3)$ Assume that $e$ is an open unit in $S$. Given any $a\in S$, we need to prove that the central shift $c_a:S\times S\to S$, $c_a:(x,y)\subset xay$, is an open map. We should prove that for any non-empty open set $W\subset S\times S$ the image $c_a(W)$ is open in $S$. Fix any point $y\in c_a(W)$ and find a pair $(v,w)\in W$ such that $y=vaw$. Find two open sets $O_v,O_w\subset S$ such that $(v,w)\in O_v\times O_w\subset W$. By the continuity of the right shift $r_v:S\to S$, $r_v:x\mapsto xv$, and the left shift $\ell_w:S\to S$, $\ell_w:x\mapsto wx$, the set $U=r_v^{-1}(O_v)\cap \ell_w^{-1}(O_w)$ is an open neighborhood of $e$. Since $e$ is an open unit, the set $UyU$ is a neighborhood of $y$. Since $UyU=UvawU=r_v(U)a\ell_w(U)\subset O_vaO_w=c_a(O_v\times O_w)\subset c_a(W)$, the point $y$ is an interior point of $c_a(W)$, which implies that the set $c_a(W)$ is open in $S$.
\smallskip

$(1)\Ra(4)$ Any open neighborhoods $U,U'\in\mathcal N_e$ of $e$ determines the entourage $$E_{U,U'}=\{(x,y)\in S\times S:y\in UxU'\}\in\mathcal L_S\wedge\mathcal R_S.$$ It is easy to see that the family $\mathcal B_{\mathcal L\wedge\mathcal R}=\{E_{U,U'}:U,U'\in\mathcal N_e\}$ is a multiplicative base of the quasi-uniformity $\mathcal L\wedge\mathcal R$. The rotundness of the base $\mathcal B$ follows from the inclusion
$$B(\overline{A};E_{V,V'})=V\overline{A}V'\subset \overline{VAV'}=\overline{B}(A;E_{V,V'})$$ holding for every non-empty subset $A\subset X$ and every entourage $E_{V,V'}\in\mathcal B_{\mathcal L\wedge\mathcal R}$.
\end{proof}

By analogy we can prove the ``left'' and ``right'' versions of Proposition~\ref{p5.1}. The equivalences (1)--(3) in these propositions have been proved earlier by Kopperman \cite[Lemma 7]{Koper}.

\begin{proposition}\label{p5.2} For a topological semigroup $S$ with a right unit $e$ the following conditions are equivalent:
\begin{enumerate}
\item the quasi-uniformity $\mathcal L$ generates the topology of $S$;
\item $e$ is an open right unit in $S$;
\item $S$ has open left shifts.
\end{enumerate}
The conditions \textup{(1)--(3)} imply:
\begin{enumerate}
\item[(4)] the quasi-uniformity $\mathcal L$ is rotund.
\end{enumerate}
\end{proposition}

\begin{proposition}\label{p5.3} For a topological semigroup $S$  with a left unit $e$ the following conditions are equivalent:
\begin{enumerate}
\item the quasi-uniformity $\mathcal R$ generates the topology of $S$;
\item $e$ is an open left unit in $S$;
\item $S$ has open right shifts;
\end{enumerate}
The conditions \textup{(1)--(3)} imply:
\begin{enumerate}
\item[(4)] the quasi-uniformity $\mathcal R$ is rotund.
\end{enumerate}
\end{proposition}

Propositions~\ref{p5.2} and \ref{p5.3} imply the following known characterization (see \cite[1.8]{KMR}):

\begin{proposition} For a topological monoid $S$ the following conditions are equivalent:
\begin{enumerate}
\item the quasi-uniformity $\mathcal L\vee \mathcal R$ generates the topology of $S$;
\item the unit $e$ of $S$ is an open left unit and an open right unit in $S$;
\item $S$ has open shifts.
\end{enumerate}
\end{proposition}

Since the topology of any topological semigroup with an open (left, right) unit is generated by a  rotund quasi-uniformity, Corollary~\ref{c4.2} implies:

\begin{corollary} Each topological semigroup $S$ with an open (left or right) unit $e$ is:
\begin{enumerate}
\item  completely regular if and only if $S$ is regular if and only if $S$ is semiregular;
\item functionally Hausdorff if and only if $S$ is Hausdorff if and only if $S$ is functionally Hausdorff.
\end{enumerate}
\end{corollary}

Discussing subinvariant quasi-pseudometrics on topological semigroups, we shall need the notion of a balanced topological monoid, which generalizes the well-known notion of a balanced topological group. Let us recall \cite[p.69]{AT} that a topological group $G$ is called {\em balanced} if its unit $e$ has a neighborhood base $\mathcal N_e$ consisting of neighborhoods $U\subset G$, which are invariant in the sense that $xUx^{-1}=U$ for every $x\in G$. Taking into account that the equality $xUx^{-1}=U$ is equivalent to $xU=Ux$, we define a subset $U\subset S$ of a semigroup $S$ to be {\em invariant} if $xU=Ux$ for all points $x\in S$.

\begin{definition}
A point $e\in S$ of a topological semigroup $S$ is called
\begin{itemize}
\item a {\em balanced point} if it has a neighborhood base $\mathcal N_x$ consisting of open invariant neighborhoods;
\item a {\em balanced open} ({\em left, right}) {\em unit} if $e$ is an open (left, right) unit in $S$ and $e$ is a balanced point in $S$.
\end{itemize}
A topological monoid $S$ will be called a {\em balanced topological monoid} if its unit $e$ is a balanced point in $S$.
\end{definition}

\begin{proposition}\label{p5.7} If $e$ is a balanced open right unit in a topological semigroup $S$, then $S$ has open left and right shifts. If $S$ is a $T_1$-space, then $e$ is an open left unit in $S$, which implies that $S$ is a balanced topological monoid.
\end{proposition}

\begin{proof} Assume that $e$ is a balanced open right unit in $S$.
By Proposition~\ref{p5.2}, the topological semigroup $S$ has open left shifts. Next, we show that $S$ has open right shifts. Given any point $a\in S$ and open set $U\subset S$ we should prove that the set $Ua$ is open in $S$. Fix any point $v\in Ua$ and find a point $u\in U$ such that $v=ua$. The continuity of the left shift $\ell_u:S\to S$  implies that the set $\ell_u^{-1}(U)=\{x\in S:ux\in U\}$ is an open neighborhood of $e$ in $S$. Since $e$ is a balanced point in $S$, there is an open (invariant) neighborhood $V_e\subset \ell^{-1}_u(U)$ of $e$ such that $V_ea=aV_e$. Taking into account that the left shift $\ell_v:S\to S$ is open, we conclude that the set $vV_e=uaV_e=uV_ea$ is open. Since $vV_e=uV_ea\subset Ua$, we see that $v$ is an interior point of the set $Ua$. So, $Ua$ is open in $S$.

Finally, assume that $S$ is a $T_1$-space. We claim that $ex=x$ for every $x\in X$. Assuming the converse, we can find a neighborhood $U_{ex}\subset S$ of the point $ex\ne x$, which does not contain the point $x$. By the continuity of the right shift $r_x:S\to S$, there exists a neighborhood $V_e\subset S$ of $e$ such that $ V_ex\subset U_{ex}$. Since $e$ is balanced, we can additionally assume that $V_ex=xV_e$, which implies that $V_ex=xV_e$ is a neighborhood of $x=xe$. But this is not possible as $V_ex\subset U_e$ does not contain the point $x$. So, $e$ is a unit in $S$. Since $S$ has open right shifts, the unit $e$ is an open left unit in $S$ according to Proposition~\ref{p5.3}.
\end{proof}

\section{Subinvariant quasi-pseudometrics on topological monoids}\label{s6}

In this section we shall apply (the full strength of) Theorem~\ref{main} to construct (left, right) subinvariant quasi-pseudometrics on topological semigroups possessing an open (right, left) unit.

A premetric $d:S\times S\to[0,\infty)$ on a semigroup $S$ is called
\begin{itemize}
\item {\em left-invariant} if $d(zx,zy)=d(x,y)$ for any points $x,y,z\in S$;
\item {\em right-invariant} if $d(xz,yz)=d(x,y)$ for any points $x,y,z\in S$;
\item {\em invariant} if $d$ is left-invariant and right invariant;   \smallskip

\item {\em left-subinvariant} if $d(zx,zy)\le d(x,y)$ for any points $x,y,z\in S$;
\item {\em right-subinvariant} if $d(xz,yz)\le d(x,y)$ for any points $x,y,z\in S$;
\item {\em subinvariant} if $d$ is left-subinvariant and right-subinvariant.
\end{itemize}

\begin{theorem}\label{t6.1} Let $S$ be a topological semigroup with an open right unit $e$. For every open neighborhood $U\subset S$ of $e$ there is a premetric $d:S\times S\to[0,1]$ on $S$ such that
\begin{enumerate}
\item $\overline{d}=\overline{d}^\circ\le d$;
\item $d$ is a left-subinvariant $\overline{\dist}$-continuous quasi-pseudometric with open balls on $S$;
\item $\overline{d}=\overline{d}^\circ$ is a left-subinvariant right-continuous and $\overline{\dist}$-continuous quasi-pseudometric on $S$;
\item $B_d(x,1)\subset xU$ and $B_{\bar{d}}(x,1)\subset \overline{xU}^\circ$ for every $x\in S$.
\item If $e$ is a balanced point in $S$, then the quasi-pseudometrics $d$ and $\overline{d}=\overline{d}^\circ$ are subinvariant.
\end{enumerate}
\end{theorem}

\begin{proof} Let $\mathcal N_e$ be the family of all open neighborhoods of the right unit $e$ in $S$ and $\mathcal{IN}_e=\{U\in\mathcal N_e:\forall x\in S\;(xU=Ux)\}$ be the family of open invariant neighborhoods of $e$. Put
$$\mathcal N_e'=\begin{cases}
\mathcal{IN}_e&\mbox{if $e$ is a balanced point in $S$},\\
\mathcal N_e&\mbox{otherwise},
\end{cases}
$$and observe that $\mathcal N'_e$ is a multiplicative neighborhood base at $e$ (the multiplicativity of $\mathcal N'_e$ means that for any neighborhoods $U,V\in\mathcal N_e'$ we get $UV\in\mathcal N_e'$).

Every neighborhood $U\in\mathcal N'_e$ induces the entourage $\mathbb E(U)=\{(x,y)\in S\times S:y\in xU\}$.
We claim that the family $\mathcal B=\{\mathbb E(U):U\in\mathcal N'_e\}$ of such entourages is a rotund multiplicative base of the left quasi-uniformity $\mathcal L$ on $S$. Indeed, for any non-empty subset $A\subset S$ and entourage $\mathbb E(V)\in\mathcal B$ we get ${B}(\overline{A};\mathbb E(V))=\overline{A}V\subset\overline{AV}=\overline{B}(A;E(V))$.
So, the base $\mathcal B$ is rotund.

It will be convenient for neighborhoods $U,V,W\in\mathcal N_e$ to denote the entourages $\mathbb E(U)$, $\mathbb E(V)$, $\mathbb E(W)$ by $\mathbb U$, $\mathbb V$, $\mathbb W$, respectively.

By Theorem~\ref{main}, for every open neighborhood $U\in\mathcal N_e$ there exists a family of entourages $\{\mathbb V_q\}_{q\in\IQ_2^1}$ such that
\begin{enumerate}
\item[(a)] $\mathbb V_q\circ \mathbb V_r\subset \mathbb V_{q+r}\subset \mathbb U\subset \mathbb V_1=S\times S$ for any rational numbers $q,r\in\IQ_2$ with $q+r<1$.
\item[(b)] $\overline{\mathbb V}_p^\circ\circ\overline{\mathbb V}^\circ_q\subset \overline{\mathbb V}_p\circ\overline{\mathbb V}_q\subset \overline{\mathbb V}^\circ_{r}\subset \overline{\mathbb U}^\circ$ for any rational numbers $p,q,r\in\IQ_2$ with $p+q<r$.
\end{enumerate}
Moreover, the premetric $$d:S\times S\to[0,1],\;\;\;d(x,y)=\inf\{r\in\IQ_2^1:(x,y)\in \mathbb V_r\}$$ and its (semi)regularization
have the following properties:
\begin{enumerate}
\item[(c)] $\overline{d}=\overline{d}^\circ\le d$;
\item[(d)] $d$ is an $\mathcal L$-uniform $\overline{\dist}$-continuous quasi-pseudometric with open balls such that $[d\/]_{<1}\subset \mathbb U$;
\item[(e)] $\overline{d}=\overline{d}^\circ$ is a $\mathcal L$-uniform right-continuous $\overline{\dist}$-continuous quasi-pseudometric on $S$ such that $[\overline{d}]_{<1}=\overline{\mathbb U}^\circ$;
\item[(f)] $\overline{d}(x,y)=\inf\{r\in\IQ_2^1:y\in \overline{xV_r}\}$ for any $x,y\in S$.
\end{enumerate}

The conditions (c)--(f) and the definitions of the entourages $\IU=\mathbb E(U)$ and $\overline{\IU}^\circ=\bigcup_{x\in S}\{x\}\times \overline{B}^\circ\kern-2pt(x;U)$ imply that the quasi-pseudometrics $d$ and $\overline{d}=\overline{d}^\circ$ have all the properties from the statements (1)--(4) except for the left-subinvariance.

Now we prove that the quasi-pseudometric $d$ is left-subinvariant (and right-subinvariant if $e$ is a balanced point in $S$). To see that $d$ is left-subinvariant, if suffices to check that $d(zx,zy)\le d(x,y)+\e$ for every points $x,y,z\in S$ and every real number $\e>0$. This inequality is trivial if $d(x,y)+\e\ge 1$. So, we assume that $d(x,y)+\e<1$. This this case the definition of the premetric $d$ guarantees that $(x,y)\in\mathbb V_r$ for some number $r\in\IQ_2$ with $r<d(x,y)+\e$. Find an open neighborhood $V_r\in\mathcal N_e$ of $e$ such that $\mathbb V_r=\mathbb E(V_r)$. Then $(x,y)\in\mathbb V_r$ is equivalent to $y\in x V_r$, which implies $zy\in zxV_r$, $(zx,zy)\in \mathbb E(V_r)=\mathbb V_r$, and finally $d(zx,zy)\le r<d(x,y)+\e$. If the right unit $e$ is balanced, then the set $V_r\in\mathcal B$ is invariant and hence $y\in xV_r$ implies $yz\in xV_rz=xzV_r$ and hence $d(xz,yz)\le r<d(x,y)+\e$.

Finally, we prove that the quasi-pseudometric $\overline{d}=\overline{d}^\circ$ is left-subinvariant (and right-subinvariant if $e$ is a balanced point in $S$). To see that $\overline{d}$ is left-subinvariant, if suffices to check that $\overline{d}(zx,zy)\le \overline{d}(x,y)+\e$ for every points $x,y,z\in S$ and every real number $\e>0$. This inequality is trivial if  $\overline{d}(x,y)+\e\ge 1$. So, we assume that $\overline{d}(x,y)+\e<1$. In this case the property (f) of the premetric $\overline{d}$ guarantees that $(x,y)\in\overline{\mathbb V}_r$ for some number $r\in\IQ_2$ with $r<\overline{d}(x,y)+\e$. Find an open neighborhood $V_r\in\mathcal N_e$ of $e$ such that $\mathbb V_r=\mathbb E(V_r)$. Then $(x,y)\in\overline{\mathbb V}_r$ is equivalent to $y\in \overline{x V_r}$. The continuity of the left shift $\ell_z:S\to S$ implies that $zy\in z\overline{xV_r}\subset \overline{zxV_r}$ and hence $\overline{d}(zx,zy)\le r<\overline{d}(x,y)+\e$. If the right unit $e$ is balanced, then the entourage $V_r\in\mathcal B$ is invariant. By the continuity of the right shift $r_z:S\to S$, we get $yz\in \overline{xV_r}z\subset \overline{xV_rz}=\overline{xzV_r}$ and hence $\overline{d}(xz,yz)\le r<\overline{d}(x,y)+\e$ by the property (f) of the premetric $\overline{d}$.
\end{proof}

For a topological space $X$ by $\chi(X)$ we denote the {\em character of $X$} i.e., the smallest cardinal $\kappa$ such that each point $x\in X$ has a neighborhood base $\mathcal N_x$  of cardinality
$|\mathcal N_x|\le\kappa$. In \cite{Koper} Kopperman proved that the topology of any topological monoid with open left shifts can be generated by a family of left-subinvariant quasi-pseudometrics. The following corollary of Theorem~\ref{t6.1} can be considered as a generalization of this Kopperman's result.

\begin{corollary}\label{c6.2} The topology of any (semiregular) topological semigroup $S$ with an open right unit is generated by a family $\mathcal D$ of  left-subinvariant $\overline{\dist}$-continuous (and right-continuous) quasi-pseudometrics of cardinality $|\mathcal D|\le \chi(S)$.
\end{corollary}

\begin{proof} Fix a neighborhood base $\mathcal N_e$ at the right unit $e$ of cardinality $|\mathcal N_e|\le\chi(S)$. By Theorem~\ref{t6.1}, for every neighborhood $U\in\mathcal N_e$ there exist a left-subinvariant $\overline{\dist}$-continuous quasi-pseudometric $dU$ with open balls on $S$ such that its regularization
$\overline{dU}=\overline{dU}^\circ$ is a left-subinvariant $\dist$-continuous and right-continuous quasi-pseudometric such that $B_{dU}(x;1)\subset xU$ and $B_{\overline{dU}}(x;1)\subset\overline{xU}^\circ$ for all $x\in S$. Then the family of left-subinvariant $\overline{\dist}$-continuous quasi-pseudometrics $\mathcal D=\{dU:U\in\mathcal N_e\}$ has cardinality $|\mathcal D|\le|\mathcal N_e|\le\chi(S)$ and generates the topology of $S$.

If the space $S$ is semiregular, then for every neighborhood $O_x\subset S$ we can find a neighborhood $U_x\subset S$ of $x$ such that $\overline{U}^\circ_x\subset O_x$ and a neighborhood $U\in\mathcal N_e$ such that $xU\subset U_x$. Then $B_{\overline{dU}}(x;1)\subset\overline{xU}^\circ\subset \overline{U}^\circ_x\subset O_x$ and hence the family $\overline{\mathcal D}=\{\overline{dU}:U\in\mathcal N_e\}$ generates the topology of $S$, has cardinality $|\overline{\mathcal D}|\le|\mathcal N_e|\le\chi(\U)$, and consists of left-subinvariant right-continuous $\overline{\dist}$-continuous quasi-pseudometrics.
\end{proof}

\begin{corollary}\label{c6.3} The topology of any first countable (semiregular) topological semigroup $S$ with an open right unit is generated by a left-subinvariant  $\overline{\dist}$-continuous (and right-continuous) quasi-pseudometric.
\end{corollary}

\begin{proof} By Corollary~\ref{c6.2}, the topology of $S$ can be generated by a countable family $\mathcal D=\{d_n\}_{n\in\w}$ of left-subinvariant $\overline{\dist}$-continuous (and right-continuous) quasi-pseudometrics with open balls. It can be shown that $$d=\max_{n\in\IN}\min\{d_n,2^{-n}\}:S\times S\to[0,\infty)$$ is a left-subinvariant $\overline{\dist}$-continuous (and right-continuous)  quasi-pseudometric with open balls generating the topology of $S$. \end{proof}

Replacing the left quasi-uniformity $\mathcal L$ in the proof of Theorem~\ref{t6.1} by the right quasi-uniformity $\mathcal R$ we can prove a ``right'' version of Theorem~\ref{t6.1}.

\begin{theorem}\label{t6.4} Let $S$ be a topological semigroup with an open left unit $e$. For every open neighborhood $U\subset S$ of $e$ there is a premetric $d:S\times S\to[0,1]$ on $S$ such that
\begin{enumerate}
\item $\overline{d}=\overline{d}^\circ\le d$;
\item $d$ is a right-subinvariant $\overline{\dist}$-continuous quasi-pseudometric with open balls on $S$;
\item $\overline{d}=\overline{d}^\circ$ is a right-subinvariant right-continuous $\overline{\dist}$-continuous quasi-pseudometric on $S$;
\item $B_d(x,1)\subset Ux$ and $B_{\overline{d}}(x,1)\subset\overline{Ux}^\circ$ for every $x\in S$.
\item If $e$ is a balanced point in $S$, then the quasi-pseudometrics $d$ and $\overline{d}=\overline{d}^\circ$ are subinvariant.
\end{enumerate}
\end{theorem}

Modifying the proofs of Corollaries~\ref{c6.2} and \ref{c6.3}, we can deduce the following corollaries from Theorem~\ref{t6.4}.

\begin{corollary}\label{c6.5} The topology of any (semiregular) topological semigroup $S$ with an open left unit is generated by a family $\mathcal D$ of  right-subinvariant $\overline{\dist}$-continuous (and right-continuous) quasi-pseudometrics of cardinality $|\mathcal D|=\chi(S)$.
\end{corollary}

\begin{corollary}\label{c6.6} The topology of any first countable (semiregular) topological semigroup $S$ with an open right unit is generated by a right-subinvariant $\overline{\dist}$-continuous (and right-continuous) quasi-pseudometric.
\end{corollary}

By analogy, the following corollaries can be derived from Theorem~\ref{t6.1}(5) or \ref{t6.4}(5).

\begin{corollary}\label{c6.7} The topology of any (semiregular) balanced topological monoid $S$ with an open unit is generated by a family $\mathcal D$ of  subinvariant $\overline{\dist}$-continuous (and right-continuous) quasi-pseudometrics of cardinality $|\mathcal D|=\chi(S)$.
\end{corollary}

\begin{corollary}\label{c6.8} The topology of any first-countable (semiregular) balanced  topological monoid $S$ with open unit is generated by a subinvariant $\overline{\dist}$-continuous (and right-continuous) quasi-pseudometric.
\end{corollary}

\section{(Left, right) invariant quasi-pseudometrics on paratopological groups}\label{s7}

In this section we apply the results of the preceding sections to paratopological groups. We recall that a {\em paratopological group} is a group $G$ endowed with a topology making the binary operation $G\times G\to G$, $(x,y)\mapsto xy$, continuous. It is clear that each paratopological group is a topological monoid with open shifts.

A paratopological group $G$ is called a {\em balanced paratopological group} if it is balanced as a topological monoid.

As a partial case of Propositions~\ref{p5.1}--\ref{p5.3}, we get:

\begin{proposition} For every paratopological group $G$ the quasi-uniformities $\mathcal L$, $\mathcal R$, and $\mathcal L\wedge\mathcal R$ are rotund.
\end{proposition}

For a balanced paratopological group we get $\mathcal L=\mathcal R=\mathcal L\vee\mathcal R=\mathcal L\wedge\mathcal R$, which implies that the quasi-uniformity $\mathcal L\vee\mathcal R$ is rotund.
We do not know if this remains true for any paratopological group.

\begin{problem} Is the quasi-uniformity $\mathcal L\vee\mathcal R$ rotund for any paratopological group? Is it always point-rotund?
\end{problem}

Since the topology of any paratopological group is generated by a rotund quasi-uniformity, Corollary~\ref{c4.2} implies the following corollary first proved by a different method in \cite{BR}.
This corollary answers \cite[Question 1.2]{Rav1}, \cite[Problem 1.3.1]{AT}, and \cite[Problem 2.1]{Tka}.

\begin{corollary} A paratopological group $G$ is:
\begin{enumerate}
\item  completely regular if and only if $S$ is regular if and only if $S$ it semiregular;
\item functionally Hausdorff if and only if $S$ is Hausdorff if and only if $S$ is semi-Hausdorff.
\end{enumerate}
\end{corollary}

Now we derive some results on quasi-pseudometrizability of paratopological groups by left-invariant quasi-pseudometrics. Observe that each (left, right) subinvariant premetric on a group is (left, right) invariant.
Because of that for paratopological groups Theorem~\ref{t6.1} and Corollary~\ref{c6.2} and \ref{c6.3} take the following form.

\begin{theorem} Let $G$ be a paratopological group. For every neighborhood $U\subset G$ of the unit $e$ there is a left-invariant $\overline{\dist}$-continuous quasi-pseudometric $d:G\times G\to[0,1]$ whose regularization $\overline{d}=\overline{d}^\circ$ is a left-invariant right-continuous $\overline{\dist}$-continuous quasi-pseudometric  such that $B_d(x;1)\subset xU$ and
$B_{\overline{d}}(x,1)\subset x\overline{U}^\circ$ for every $x\in G$. If the paratopological group $G$ is balanced, then the quasi-pseudometrics $d$ and $\overline{d}=\overline{d}^\circ$ are invariant.
\end{theorem}

\begin{corollary} The topology of any (semiregular) paratopological group $S$ is generated by a family $\mathcal D$ of  left-invariant $\overline{\dist}$-continuous (and right-continuous) quasi-pseudometrics of cardinality $|\mathcal D|=\chi(G)$.
\end{corollary}

\begin{corollary}\label{c7.6} The topology of any first countable (semiregular) paratopological group $G$ is generated by a left-invariant $\overline{\dist}$-continuous (and right-continuous) quasi-pseudometric.
\end{corollary}

Those ``left'' results have their ``right'' versions, which are partial cases of Theorem~\ref{t6.4} and Corollaries~\ref{c6.5}, \ref{c6.6}.

\begin{theorem} Let $G$ be a paratopological group. For every neighborhood $U\subset G$ of the unit $e$ there is a right-invariant $\overline{\dist}$-continuous quasi-pseudometric $d:G\times G\to[0,1]$ whose regularization $\overline{d}=\overline{d}^\circ$ is a right-invariant right-continuous $\overline{\dist}$-continuous quasi-pseudometric  such that $B_d(x;1)\subset xU$ and
$B_{\overline{d}}(x,1)\subset x\overline{U}^\circ$ for every $x\in G$. If the paratopological group $G$ is balanced, then the quasi-pseudometrics $d$ and $\overline{d}=\overline{d}^\circ$ are invariant.
\end{theorem}

\begin{corollary} The topology of any (semiregular) paratopological group $S$ is generated by a family $\mathcal D$ of  right-invariant $\overline{\dist}$-continuous (and right-continuous) quasi-pseudometrics of cardinality $|\mathcal D|=\chi(G)$.
\end{corollary}

\begin{corollary} The topology of any first countable (semiregular) paratopological group $G$ is generated by a right-invariant $\overline{\dist}$-continuous (and right-continuous) quasi-pseudometric.
\end{corollary}

The following two corollaries are partial cases of Corollaries~\ref{c6.7} and \ref{c6.8}.

\begin{corollary} The topology of any (semiregular) balanced paratopological group $S$ is generated by a family $\mathcal D$ of invariant $\overline{\dist}$-continuous (and right-continuous) quasi-pseudometrics of cardinality $|\mathcal D|=\chi(G)$.
\end{corollary}

\begin{corollary} The topology of any first countable (semiregular) balanced paratopological group $G$ is generated by an invariant $\overline{\dist}$-continuous (and right-continuous) quasi-pseudometric.
\end{corollary}

\begin{remark} Corollary~\ref{c7.6} answers affirmatively Question 3.1 of \cite{Rav1}.
\end{remark}

The right-continuity and $\overline{\dist}$-continuity of the invariant quasi-pseudometrics in the above results cannot be improved to the $\dist$-continuity because of the following example. We recall that the {\em Sorgenfrey topology} on the real line is generated by the base consisting of half-open intervals $[a,b)$. It is clear that the real line endowed with the Sorgenfrey topology is a first countable paratopological group.

\begin{example} The Sorgenfrey topology on the real line $\IR$ cannot be generated by an invariant $\dist$-continuous quasi-pseudometric.
\end{example}

\begin{proof} Assume that the Sorgenfrey topology $\tau_S$ on $\IR$ is generated by some invariant $\dist$-continuous quasi-pseudometric $d$. Then the neighborhood $[0,1)\in\tau_S$ of zero contains an open ball $B_d(0,\e)$. Observe that for the set $A=(0,1)$ we get $B_d(A;\e)=A+B_d(0,\e)\subset A+[0,1)=(0,2)$, which implies that $d_A(0)\ge \e>0$. On the other hand, the $\dist$-continuity of $d$ implies that the set $d^{-1}_A(\{0\})\supset A$ is closed and hence $0\in\bar A\subset d^{-1}_A(\{0\})$ and hence $d_A(0)=0$, which is a desired contradiction.
\end{proof}

The right-continuity of the quasi-pseudometrics in the above results also cannot be improved to the separate-continuity because of the following observation. We shall say that a premetric $d$ on a group $G$ with the unit $e$ is {\em weakly invariant} if $d(e,x)=d(x^{-1},e)$ for any $x\in G$. It is clear that a premetric $d$ is weakly invariant if it is left-invariant or right-invariant. The following proposition generalizes a result of Liu \cite[2.1]{Liu}.

\begin{proposition} A paratopological group $G$ is a topological group if and only if the topology of $G$ is generated by a family of left-continuous weakly invariant premetrics.
\end{proposition}

\begin{proof} The ``only if'' part is well-known. To prove the ``if'' part, assume that the topology of a paratopological group $G$ is generated by the family $\mathcal D$ of all weakly invariant left-continuous premetrics on $G$.

Given any neighborhood $U\subset G$ of the unit $e$, we should find a neighborhood $V\subset G$ of $e$ such that $V^{-1}\subset U$. Find a left-continuous weakly invariant premetric $d\in\mathcal D$ such that $B_d(e,\delta)\subset U$ for some $\delta>0$. By the left-continuity of $d$, there is a neighborhood $V\subset G$ of $e$ such that $d(v,e)<\delta$ for all $v\in V$. The weak invariance of $d$ guarantees that for every $v\in V$ we get $B_d(e,v^{-1})=B_d(v,e)<\delta$ and hence $v^{-1}\in B_d(e,\delta)\subset U$. Therefore, $V$ is a required neighborhood of $e$ with $V^{-1}\subset U$, which proves the continuity of the inversion operation $(\cdot)^{-1}:G\to G$, $(\cdot)^{-1}:x\mapsto x^{-1}$.
So, $G$ is a topological group.
\end{proof}

\end{document}